\documentclass{article}
\usepackage[english]{babel}
\input xy
\xyoption{all}
\usepackage{amssymb,amsmath,amsthm,multirow}

\newcommand{\R}{\mathbb{R}}
\newcommand{\lie}[1]{\mathfrak{#1}}     
\newcommand{\Lie}{\mathcal{L}}      

\newcommand{\hook}{\lrcorner\,}

\newcommand{\SU}{\mathrm{SU}}

\newcommand{\Gtwo}{\mathrm{G}_2}

\newcommand{\GL}{\mathrm{GL}}

\newcommand{\dfn}[1]{\emph{#1}}

\theoremstyle{plain}
\newtheorem{proposition}{Proposition}

\newtheorem{theorem}[proposition]{Theorem}
\newtheorem{lemma}[proposition]{Lemma}

\theoremstyle{definition}

\theoremstyle{remark}
\newtheorem*{remark}{Remark}

\newcommand{\Span}[1]{\operatorname{Span}\left\{#1\right\}}

\begin{document}
\title{Nilmanifolds with a calibrated $\Gtwo$-structure}
\author{Diego Conti and Marisa Fern\'andez}
\maketitle

\begin{abstract}
We introduce obstructions to the existence of a calibrated $\Gtwo$-structure
on a Lie algebra $\lie{g}$ of dimension seven, not necessarily
nilpotent. In particular, we prove that
if there is a Lie algebra epimorphism
from $\lie{g}$ to a
six-dimensional Lie algebra $\lie{h}$ with kernel contained in the center of $\lie{g}$, then
$\lie{h}$ has a symplectic form.
As a consequence, we obtain a classification of the
nilpotent Lie algebras that admit
a calibrated $\Gtwo$-structure.
\end{abstract}

\vskip5pt{\small\textbf{MSC classification}: Primary 53C38;
Secondary 53C15, 17B30}\vskip0.5pt
\vskip4pt{\small\textbf{Key words}: calibrated $G_2$ forms,
nilpotent Lie algebras, Lefschetz property}\vskip10pt

\section{Introduction}
A Riemannian $7$-manifold with holonomy contained in $\Gtwo$ can be characterized by the existence of an associated parallel $3$-form. The first examples of complete metrics
with holonomy $\Gtwo$ were given by Bryant and Salamon in~\cite{BrS}, and the first examples of compact manifolds with such a metric were constructed by Joyce in~\cite{Joyce}. Explicit examples on solvable Lie groups were constructed in~\cite{ChiossiFino}; examples on nilpotent Lie groups can be obtained by taking a nilpotent six-dimensional group with a half-flat structure and solving Hitchin's evolution equations (see \cite{Conti:HalfFlat,Hitchin:StableForms}). More generally, one can consider $\Gtwo$-structures where the associated $3$-form $\varphi$ is closed: then $\varphi$ defines a calibration (\cite{HarveyLawson}), and the $\Gtwo$-structure is  said to be \dfn{calibrated}. An equivalent condition is that the intrinsic torsion
 lies in the component $\mathcal{X}_2\cong\lie{g}_2$ (\cite{FernandezGray}).

Compact calibrated $\Gtwo$ manifolds have interesting curvature properties. It is well known that a $\Gtwo$ holonomy manifold is Ricci-flat, or equivalently, both Einstein and scalar-flat. On a compact calibrated $\Gtwo$ manifold, both the Einstein condition (\cite{CleytonIvanov}) and scalar-flatness (\cite{Bryant}) are equivalent to the holonomy being contained in $\Gtwo$. In fact, \cite{Bryant} shows that the scalar curvature is always non-positive.

Constructing examples is not a straightforward task. For instance, \cite{CleytonSwann} classifies calibrated $\Gtwo$-manifolds on which a simple groups acts with cohomogeneity one, and  no compact manifold occurs in this list. On the other hand, the second author exhibited the first example of a compact calibrated $\Gtwo$-manifold that does not have holonomy $\Gtwo$ \cite{Fernandez}.
This example is given in terms of a nilpotent Lie algebra $\lie{g}$ and an element of $\Lambda^3\lie{g}^*$ that corresponds to a closed left-invariant $3$-form on the associated simply-connected Lie group. Since the structure constants are rational, there exists a uniform discrete
subgroup  \cite{Malcev}; the quotient, called a nilmanifold, has an induced calibrated $\Gtwo$-structure.

\smallskip
In this paper we pursue this approach, and classify the nilpotent $7$-dimensional Lie algebras that admit a calibrated $\Gtwo$-structure. Since the structure constants turn out to be rational, each Lie algebra determines a nilmanifold. So, we obtain 12 compact calibrated $\Gtwo$ nilmanifolds (see  Theorem~\ref{G2-calibrated-decomposable}, Lemma~\ref{lemma:withparameter} and Lemma~\ref{lemma:noparameters}). Three of them are reducible: they are the product
of a circle with a symplectic half-flat nilmanifold, the latter being classified in \cite{ContiTomassini}. The remaining nine are new.

The proof is based on two necessary conditions that a Lie algebra must satisfy for a calibrated $\Gtwo$-structure to exist (see Proposition~\ref{prop:epi} and Lemma~\ref{lemma:obstr2}).

Our first obstruction is related to a construction of \cite{ApostolovSalamon}. Suppose that $M$ is a $7$-manifold with a calibrated $\Gtwo$ form $\varphi$, and $X$ is a unit Killing field, i.e. $\Lie_X\varphi=0$. Then if $\eta=X^\flat$, we can write
\[\varphi=\omega\wedge\eta+\psi^+,\]
where $\omega$, $\psi+$ and $d\eta$ are basic forms with respect to the $1$-dimensional foliation defined by $X$. Suppose in addition that $X$ is the fundamental vector field of a free $S^1$ action; then basic forms can be identified with forms on $M/S^1$. By
\[0=\Lie_X\varphi=d(X\hook \varphi) =d\omega,\]
$\omega$ is a symplectic form on $M/S^1$. Moreover
\[0=\omega\wedge d\eta +d\psi^+\]
implies that $[d\eta]$ is in the kernel of the map
\[H^2(M/S^1)\to H^4(M/S^1), \quad [\beta]\to [\beta\wedge\omega].\]
If this map is an isomorphism, then the $S^1$-bundle is trivial: this puts topological restrictions on $M$, which translate to algebraic conditions in our setup. A similar method was used in \cite{ContiTomassini}.

In principle, these restrictions reduce our problem to the classification of symplectic nilpotent Lie algebras of dimension six for which the map $H^2\to H^4$ is non-injective (see the remark before Lemma~\ref{lemma:obstr2}). The complexity of the required calculations, however, motivate a different approach. In analogy with \cite{Conti:HalfFlat}, we introduce a second obstruction, that
requires computing the space of closed $3$-forms. It consists in the observation that $(X\hook\varphi)^3$ must be nonzero, whenever $X$ is a nonzero vector and $\varphi$ a $3$-form defining a $\Gtwo$-structure.

The final ingredient is Gong's classification of $7$-dimensional indecomposable nilpotent Lie algebras (\cite{Gong}). This list contains $140$ Lie algebras and $9$ one-parameter families; in addition, there are $35$ decomposable nilpotent Lie algebras (\cite{Magnin}, \cite{Salamon}). Calculations on a case-by-case basis  show that our list of $12$ examples is complete.

\section{Calibrated $\Gtwo$-structures and obstructions}
\label{sec:obstructions}
In this section we show
obstructions to the existence of a calibrated $\Gtwo$ form on a Lie algebra
(not necessarily nilpotent).  First, we recall
some definitions and results about $G_2$-structures.

Let us consider the space $\mathbb {O}$ of the Cayley numbers, which
is a non-associative algebra over $\R$ of dimension $8$. Thus, we can identify
$\R^7$ with the subspace of  $\mathbb {O}$ consisting of pure imaginary Cayley numbers. Then,
the product on $\mathbb {O}$ defines on  $\R^7$
the $3$-form given by
\begin{equation}
 \label{eqn:Gtwoform}
 e^{127}+e^{347}+e^{567}+e^{135} -e^{236}-e^{146}-e^{245}
\end{equation}
(see \cite{FernandezGray} and \cite{Gray} for details), where $\{e^1,\dotsc, e^7\}$ is the standard basis of $(\R^7)^*$.
Here, $e^{127}$ stands for $e^1\wedge e^2\wedge e^7$, and so on.
The group $G_2$ is the stabilizer of \eqref{eqn:Gtwoform}
under the standard action of $\GL(7,\R)$ on $\Lambda^3(\R^7)^*$.
$G_2$ is one of the exceptional Lie groups, and it is a compact, connected, simply connected
simple Lie subgroup of $SO(7)$ of dimension $14$.

A $\Gtwo$-structure on a $7$-dimensional Riemannian manifold $(M,g)$ is a reduction
of the structure group $O(7)$ of the frame bundle to $G_2$. Manifolds admitting
a $\Gtwo$-structure are called $\Gtwo$ manifolds. The existence of a
$\Gtwo$-structure on $(M,g)$ is determined by a global $3$-form $\varphi$ (the $G_2$ form)
which can be locally written as \eqref{eqn:Gtwoform} with respect to some (local)
 basis $\{e^1,\dotsc, e^7\}$ of the (local) $1$-forms on $M$.
We say that the $\Gtwo$ manifold $M$ has a {\em calibrated $G_2$-structure} if there is a $G_2$-structure on $M$ such that the $3$-form
$\varphi$ is closed, and so
$\varphi$ defines a calibration \cite{HarveyLawson}.

If $G$ is a $7$-dimensional Lie group with Lie algebra $\lie{g}$, then a $\Gtwo$-structure on $G$ is left-invariant if and only if the corresponding
$3$-form is left-invariant. Thus, a left invariant $\Gtwo$-structure on $G$ corresponds to an element $\varphi$ of $\Lambda^3\lie{g}^*$ that can be written as \eqref{eqn:Gtwoform} with respect to some coframe $\{e^1,\dotsc, e^7\}$ on $\lie{g}^*$; and we shall say that $\varphi$ defines a $\Gtwo$-structure on $\lie{g}$.  We say that a $\Gtwo$-structure on $\lie{g}$ is {\em calibrated} if $\varphi$ is closed, i.e.
\[
d\varphi=0,\]
where $d$ denotes the Chevalley-Eilenberg differential on $\lie{g}^*$. If $\Gamma$ is a discrete subgroup of $G$, a $\Gtwo$-structure on $\lie{g}$ induces a
$\Gtwo$-structure on the quotient $\Gamma\backslash G$.
Moreover, in \cite{Malcev} it is proved that if $\lie{g}$ is nilpotent with rational structure
constants, then the associated simply connected Lie group $G$ admits a uniform discrete
subgroup $\Gamma$. Therefore, a $\Gtwo$-structure on $\lie{g}$ determines
a $\Gtwo$-structure on the compact manifold $\Gamma\backslash G$, which is called a
compact nilmanifold; and if $\lie{g}$ has a calibrated
$\Gtwo$-structure, the $\Gtwo$-structure on $\Gamma\backslash G$ is also calibrated.

\smallskip
In order to show obstructions to the existence of a calibrated $\Gtwo$ form on a Lie algebra
$\lie{g}$, let us consider first a direct sum $\lie{g}=\lie{h}\oplus\R$. If $\varphi$ is a $\Gtwo$ form on $\lie{g}$, and the decomposition is orthogonal with respect to the underlying metric, then
\[\varphi=\omega\wedge \eta+\psi^+,\]
where $\omega,\psi^+$ are forms on $\lie{h}$ and $\eta$ generates
the dual of the ideal $\R$. The pair $(\omega,\psi^+)$ defines an $\SU(3)$-structure on $\lie{h}$. The condition that $\varphi$ is closed is equivalent to both $\omega$ and $\psi^+$ being closed; this means that the $\SU(3)$-structure is \dfn{symplectic half-flat}. There are exactly three nilpotent Lie algebras of dimension six that admit a symplectic half-flat structure, classified in~\cite{ContiTomassini}. So, if we focus our attention on decomposable nilpotent Lie algebras, there are at least three $7$-dimensional Lie algebras with a calibrated $\Gtwo$-structure; we will see that these are all.

More generally, every $7$-dimensional nilpotent Lie algebra fibres over a nilpotent Lie algebra of dimension six. In fact if $\xi$  is in the center of $\lie{g}$, then the quotient ${\lie{g}}/{\Span{\xi}}$ has a unique Lie algebra structure that makes the projection map
\[\lie{g}\to \frac{\lie{g}}{\Span{\xi}}\]
a Lie algebra morphism. Moreover, due to the nilpotency assumption every epimorphism $\lie{g}\to\lie{h}$, with $\lie{h}$ of dimension six, is of this form.
Using the pullback, we can identify forms on the quotient with {\em basic forms} on $\lie{g}$; in this
setting, $\alpha$ is {\em basic} if  $\xi\hook\alpha=0$.

Given a $\Gtwo$-structure on $\lie{g}$ with associated $3$-form $\varphi$ and a nonzero vector $\xi$ in the center, let $\eta=\xi^\flat$; then we can write
\[\varphi= \omega\wedge\eta+\psi^+, \quad \xi\hook\omega=0=\xi\hook\psi^+,\]
and up to a normalization coefficient the forms $(\omega,\psi^+)$ define an $\SU(3)$-structure on the six-dimensional quotient (see also~\cite{ApostolovSalamon}). In analogy with the case of a circle bundle, we shall think of $\eta$ as a connection form, and $d\eta$ as the curvature.
\begin{proposition}
\label{prop:epi}
Let $\lie{g}$ be a $7$-dimensional Lie algebra with a calibrated $\Gtwo$-structure and a non-trivial center. If $\pi \colon \lie{g}\to\lie{h}$ is a Lie algebra epimorphism with kernel contained in the center, and $\lie{h}$ of dimension six, then
$\lie{h}$ admits a symplectic form $\omega$, and the curvature form is in the kernel of
\begin{equation}
 \label{eqn:Lefschetz}
H^2(\lie{h^*}) \xrightarrow{\cdot\wedge\omega} H^4(\lie{h^*}).
\end{equation}
If the curvature form is exact on $\lie{h}$, then $\lie{g}\cong\lie{h}\oplus\R$ as Lie algebras.
\end{proposition}
\begin{proof}
Write
\[\varphi=\pi^*\omega\wedge\eta+\pi^*\psi^+\]
where $(\omega,\psi^+)$ are forms on $\lie{h}$. Since $d$ commutes with the pullback,
\[0=d\varphi=d\pi^*\omega\wedge\eta+ \pi^*\omega\wedge d\eta + \pi^*d\psi^+,\]
where $\pi^*d\omega$, $d\eta$ and $\pi^*d\psi^+$ are basic. Thus $\omega$ is a symplectic form and $d\eta$ is in the kernel of \eqref{eqn:Lefschetz}.

Now suppose that $d\eta$ is exact on $\lie{h}$. Then,
the epimorphism $\pi \colon \lie{g}\to\lie{h}$ is trivial, that is $\lie{g}=\lie{h}\oplus\R$.
More precisely, we can choose a different, closed connection form $\tilde\eta$, and $\lie{g}=\ker\tilde\eta\oplus\ker\pi$ is a direct sum of Lie algebras; by construction, $\ker\tilde\eta$ is isomorphic to $\lie{h}$.
\end{proof}

\smallskip
\begin{remark}
In the previous Proposition, we must notice that when the
curvature form is zero, $(\omega, \psi^+)$ is a symplectic half-flat structure
on $\lie{h}$. Therefore, if  $\lie{h}$ is nilpotent, by \cite{ContiTomassini} $\lie{h}$ is one of
\[(0,0,0,0,0,0), \quad (0,0,0,0,12,13), \quad (0,0,0,12,13,23).\]
With notation from \cite{Salamon},$(0,0,0,0,12,13)$ represents a the Lie algebra with a fixed basis $e^1,\dotsc e^6$ of $\lie{g}^*$, satisfying
\[de^1=0=de^3+de^3=de^4, \quad de^5=e^{12},de^6=e^{13}.\]
\end{remark}
\smallskip
\begin{remark}
Another obstruction to the existence of a calibrated $\Gtwo$-structure on a nilpotent Lie algebra is given by the condition $b_3>0$. Indeed, if $\varphi$ is a closed  $\Gtwo$ form on a nilpotent Lie algebra $\lie{g}$, and $X$ is a nonzero vector in the center of $\lie{g}$, then $\Lie_X\varphi=0$, so $X\hook\varphi$ is closed. If $\varphi$ were exact, say $\varphi=d\beta$, then the $7$-form
\[(X\hook\varphi)\wedge(X\hook\varphi)\wedge\varphi=d((X\hook\varphi)\wedge(X\hook\varphi)\wedge\beta)\]
would also be exact, hence zero, which is absurd. On the other hand, $b_3$ is always positive on a nilpotent Lie algebra of dimension seven.
\end{remark}
\smallskip
Proposition~\ref{prop:epi} motivates the following definition. We say that a $6$-dimensional Lie algebra $\lie{h}$ satisfies the \dfn{2-Lefschetz property} if, for every symplectic structure on $\lie{h}$, the map \eqref{eqn:Lefschetz} is an isomorphism. This condition holds trivially when $\lie{h}$ has no symplectic structure, namely when $\lie{h}$ is one of
\begin{equation*}
 \label{nonsymplectic}
\begin{aligned}
&(0, 0, 0, 12, 23, 14 + 35), &&( 0, 0, 0, 12, 23, 14 - 35),\\
&(0, 0, 0, 12, 13, 14 + 35),  &&(0, 0, 0, 0, 12, 15 + 34),\\
&(0, 0, 0, 0, 0, 12 + 34),      &&(0,0,12,13,14+23,34+52), \\
&(0, 0, 12, 13, 14, 34 + 52), &&(0,0,0,12,14,24). \\
\end{aligned}
\end{equation*}

\smallskip
It is well known~\cite{BensonGordon} that if $(\lie{h},\omega)$ is a
$6$-dimensional, symplectic nilpotent Lie algebra, the map
\begin{equation*}
H^1(\lie{h^*}) \xrightarrow{\cdot\wedge\omega^2} H^5(\lie{h^*}).
\end{equation*}
is not surjective. However, in the next proposition, we prove that some of those Lie algebras
satisfy the 2-Lefschetz property.

\smallskip
\begin{proposition}
\label{prop:2Lefschetz}
Among 6-dimensional nilpotent Lie algebras with a symplectic structure, those that satisfy the 2-Lefschetz property are
\[(0,0,0,0,0,0);  \quad (0, 0, 12, 13, 23, 14); \quad (0, 0, 12, 13, 23, 14 + 25).\]
\end{proposition}
\begin{proof}
In the abelian case, the bilinear map \[H^2\otimes H^2\to H^4\] induced by the wedge product is non-degenerate, in the sense that for every nonzero $\beta\in H^2$, the induced linear map $\cdot\wedge\beta\colon H^2\to H^4$ is an isomorphism.

For the second Lie algebra, the cohomology class of a generic symplectic form is represented by \[\omega= \lambda_1 e^{16} +   \lambda_2 {(e^{15}+e^{24})}+ \lambda_3 e^{25}+ \lambda_4 {(e^{34}-e^{26})};\]
non-degeneracy implies $\lambda_4\neq0$. The map $H^2\to H^4$ of \eqref{eqn:Lefschetz} is represented by the matrix
\[\left(\begin{array}{cccc}\lambda_3&2  \lambda_4&\lambda_1&2  \lambda_2\\\lambda_4&0&0&\lambda_1\\0&0&0&-2  \lambda_4\\0&0&\lambda_4&\lambda_3\end{array}\right)\]
which is invertible by the assumption $\lambda_4\neq0$.

Similarly, for the last Lie algebra
\[\omega= \lambda_1e^{14}+  \lambda_2{(e^{15}+e^{24})}  - \lambda_3 {(e^{26}-e^{34})}+ \lambda_4 {(e^{16}+e^{35})}.\]
The map  \eqref{eqn:Lefschetz}  is represented by
\[\left(\begin{array}{cccc}\lambda_3&2  \lambda_4&\lambda_1&2  \lambda_2\\- \lambda_4&2  \lambda_3&2  \lambda_2&- \lambda_1\\0&0&-2  \lambda_3&2  \lambda_4\\0&0&- \lambda_4&- \lambda_3\end{array}\right),
\]
which is invertible unless
$\lambda_4^2+\lambda_3^2=0$,
which makes $\omega$ degenerate.

For all but three of the remaining Lie algebras, we observe that the bilinear map \[H^2\otimes H^2\to H^4\] is degenerate in the sense that, for every nonzero $\beta\in H^2$, the map \[\alpha\to\alpha\wedge\beta,\quad H^2\to H^4\]
is non-injective. The three exceptions are
\[(0, 0, 12, 13, 23, 14 - 25), \quad (0, 0, 0, 12, 13, 23), \quad (0,0,0,0,0,12).\]
However, either Lie algebra has a symplectic form that makes the map \eqref{eqn:Lefschetz} non-injective.
In fact, on the Lie algebra $\lie{h}$ defined by the equations $(0, 0, 12, 13, 23, 14 - 25)$,
consider the symplectic form
$$
\omega=- e^{16}+e^{15}+e^{35}+e^{34}+e^{24}- e^{26}.
$$
Then one can check
that $e^{14}+e^{25}+e^{15}+e^{24}$ defines a non-trivial class in $H^2(\lie{h^*})$, but
$$
(e^{14}+e^{25}+e^{15}+e^{24})\wedge\omega=2 e^{1245}=2d(e^{146}).
$$
Now, on the Lie algebra $(0, 0, 0, 12, 13, 23)$ we consider the symplectic form
$\omega=e^{14}+e^{26}+e^{35}$. Then,
$$
(-e^{15}-e^{24}+e^{36})\wedge\omega=d(e^{456});
$$
finally, on the Lie algebra $(0,0,0,0,0,12)$,
\[(e^{16}+e^{25}+e^{34})\wedge e^{13} =-de^{356}.\qedhere\]
\end{proof}

In principle, one could try to classify all pairs $(\lie{h},\omega)$, with $\lie{h}$ nilpotent of dimension six and $\omega$ a symplectic form on $\lie{h}$, for which \eqref{eqn:Lefschetz} is non-injective. This means that
$\omega\wedge \gamma=d\psi^+$, for some $\psi^+\in\Lambda^3\lie{h}^*$ and
some closed non-exact $2$-form $\gamma\in\Lambda^2\lie{h}^*$. If in addition, $(\omega,\psi^+)$ are compatible in the sense that they define an $\SU(3)$-structure, then declaring $de^7=\gamma$ one obtains a $7$-dimensional Lie algebra $\lie{g}$ with a calibrated $\Gtwo$-structure. By Proposition~\ref{prop:epi}, all calibrated $\Gtwo$-structures on indecomposable nilpotent Lie algebras are obtained in this way.

However, these calculations turn out to be difficult (although in one dimension less a similar approach was pursued succesfully in \cite{ContiTomassini}), and for this reason we shall use a different method
(see Section~\ref{sec:classification}),
starting with Gong's classification of $7$-dimensional Lie algebras. In fact, given a Lie algebra, it is straightforward to compute the space of its closed $3$-forms. In the spirit of~\cite{Conti:HalfFlat}, the existence of a calibrated $\Gtwo$-structure puts restrictions on this space. Whilst straightforward, the following result turns out to give an effective obstruction.
\begin{lemma}
\label{lemma:obstr2}
Let $\lie{g}$ be a $7$-dimensional nilpotent Lie algebra. If there is a nonzero $X$ in $\lie{g}$ such that $(X\hook \phi)^3=0$ for every closed $3$-form on $\lie{g}$, then $\lie{g}$ has no calibrated $\Gtwo$-structure.
\end{lemma}
\begin{proof}
Obvious.
\end{proof}
\begin{remark}
When $\lie{g}$ fibers over a non-symplectic Lie algebra $\lie{h}$, this obstruction is satisfied
automatically. Indeed, suppose $\pi\colon \lie{g}\to\lie{h}$ is a Lie algebra epimorphism; then any
closed $3$-form on $\lie{g}$ can be written as
\[\pi^*\omega\wedge\eta+\pi^*\psi^+,\]
as in the proof of Proposition~\ref{prop:epi}. So $\omega$ is a closed form on $\lie{h}$; if we assume $\lie{h}$ has no symplectic form, then $\omega^3=0$. Then the condition of Lemma~\ref{lemma:obstr2} is satisfied with $X$ a generator of $\ker\pi$.
\end{remark}

\section{Decomposable case}
In this section we classify the decomposable nilpotent Lie algebras with a calibrated $\Gtwo$-structure. Indeed, we prove:

\begin{theorem}\label{G2-calibrated-decomposable}
Among the $35$ decomposable nilpotent Lie algebras of dimension $7$,
those that have a calibrated $\Gtwo$-structure are
\[(0,0,0,0,0,0,0), \quad (0,0,0,0,12,13,0), \quad (0,0,0,12,13,23,0).\]
\end{theorem}
\begin{proof}
By the remark at the beginning of Section~\ref{sec:obstructions}, we know that these three Lie algebras have
a calibrated $\Gtwo$-structure (see \cite{Fernandez} where the second of these Lie algebras was considered).
In fact, on the non-abelian Lie algebras $(0,0,0,0,12,13)$ and
$(0,0,0,12,13,23)$ we can consider the symplectic half flat structure $(\omega_1, \psi_1^+)$ and
$(\omega_2, \psi_2^+)$, respectively, defined by
$$
\omega_1 = e^{14} + e^{26} + e^{35},  \quad  \psi_1^+ = e^{123} + e^{156} + e^{245} - e^{346},
$$
 and
 $$
\omega_2 = e^{16} + 2e^{25} + e^{34},  \quad  \psi_2^+ = e^{123} + e^{145} + e^{246} - e^{356}.
$$
Using Lemma \ref{lemma:obstr2}, we can see that the decomposable Lie algebra
\[0,0,0,0,12,34,36\]
has no calibrated $\Gtwo$-structure. Indeed a basis of the space $Z^3$ of the closed
$3$-forms is given by
\begin{multline*}e^{123},e^{124},e^{125},e^{134},e^{135},e^{136},e^{137},e^{145},e^{146},e^{234},e^{235},e^{236},\\
e^{237},e^{245},e^{246},-e^{126}+e^{345},e^{346},e^{347},e^{127}+e^{356},e^{367},e^{467}.
\end{multline*}
Thus $e_7\hook Z^3$ is the span of $e^{13}, e^{23},e^{34},e^{12},e^{36},e^{46}$, which contains
only degenerate forms.

Since this is the only decomposable nilpotent Lie algebra of dimension seven which does not have
the form $\lie{h}\oplus\R$,
it remains to prove that if $\lie{g}=\lie{h}\oplus\R$ has a calibrated $\Gtwo$ form, then $\lie{g}$ must be as in the statement.

Clearly, if $\lie{h}$ is one of the
eight Lie algebras defined by  (\ref{nonsymplectic}),
Proposition  \ref{prop:epi} implies that  the Lie algebra $\lie{g}=\lie{h}\oplus\R$ has no
calibrated $\Gtwo$ form.

Also, one can check that
none of the
five Lie algebras defined by
\begin{gather*}
(0,0,12,13,23,14,0), \quad (0,0,12,13,23,14 + 25,0), \quad (0,0,12,13,23,14 - 25,0),\\
\quad (0,0,0,0,13 + 42,14 + 23,0),  \quad (0,0,0,0,12,14 + 23,0),
\end{gather*}
has a calibrated $\Gtwo$ form because each of these is a
bundle over a non-symplectic
Lie algebra of dimension six. Explicitly, the base of the bundle and curvature form are given by
\begin{gather*}
\pi \colon (0,0,12,13,23,14,0)\to (0,0,12,13,23,0),   \quad   d\eta = e^{14},  \quad\\
\pi \colon (0,0,12,13,23,14 + 25,0)\to (0,0,12,13,23,0),   \quad   d\eta = e^{14}+e^{25},  \quad\\
\pi \colon (0,0,12,13,23,14 - 25,0)\to (0,0,12,13,23,0),   \quad   d\eta = e^{14}-e^{25},  \quad\\
\pi \colon (0,0,0,0,13 + 42,14 + 23,0)\to (0,0,0,0,13 + 42,0),   \quad   d\eta = e^{14}+e^{23},  \quad\\
\pi \colon (0,0,0,0,12,14 + 23,0)\to (0,0,0,0,14 + 23,0),   \quad   d\eta = e^{12}.
\end{gather*}
For each of the remaining $18$
Lie algebras of the form $\lie{g}=\lie{h}\oplus\R$, listed in Table~\ref{table:decomp} alongside with a basis of the space of closed $3$-forms, one can check that the hypothesis of Lemma~\ref{lemma:obstr2} is satisfied with $X=e_6$.
\end{proof}
\begin{table}
 \caption{\label{table:decomp} Closed $3$-forms on decomposable Lie algebras}
\[\begin{array}{|ll|}
\hline
(0,0,12,13,14+23,24+15,0) & e^{123},e^{124},e^{125},e^{126},e^{127},e^{134},e^{135},e^{137},e^{136}+e^{145},e^{146},e^{147},\\
\multicolumn{2}{|c|}{e^{234},e^{235}+e^{136},e^{237},-e^{236}+e^{245},e^{157}+e^{247},e^{167}+e^{257},-\frac{1}{2} e^{156}+e^{345}-\frac{1}{2} e^{246},e^{167}+e^{347}}\\
(0, 0, 0, 12, 14, 15 + 23,0) & e^{123},e^{124},e^{125},e^{126},e^{127},e^{134},e^{135},e^{136},e^{137},e^{145},e^{147},e^{157},\\
\multicolumn{2}{|c|}{e^{234},e^{235}-e^{146},-e^{156}+e^{236},e^{237},e^{245},e^{247},e^{156}+e^{345},-e^{167}+e^{347},-e^{267}+e^{457}}\\
(0, 0, 0, 12, 14 - 23, 15 + 34,0) & e^{123},e^{124},e^{125},e^{127},e^{134},e^{135},e^{136},e^{137},e^{126}+e^{145},e^{146},e^{147},\\
\multicolumn{2}{|c|}{e^{234},e^{235}-e^{126},e^{237},e^{245},e^{247},e^{236}+e^{345},e^{347}+e^{157},e^{357}+e^{167}}\\
(0, 0, 0, 12, 14, 15,0) & e^{123},e^{124},e^{125},e^{126},e^{127},e^{134},e^{135},e^{136},e^{137},e^{145},e^{146},e^{147},\\
\multicolumn{2}{|c|}{e^{156},e^{157},e^{167},e^{234},e^{237},e^{245},e^{247},e^{345}+e^{236},-e^{267}+e^{457}}\\
(0, 0, 0, 12, 13, 14 + 23,0) & e^{123},e^{124},e^{125},e^{126},e^{127},e^{134},e^{135},e^{136},e^{137},e^{145},e^{147},e^{157},\\
\multicolumn{2}{|c|}{e^{234},e^{235},e^{236}-e^{146},e^{237},e^{245}+e^{146},e^{246},e^{247},e^{257}+e^{167},e^{345}+e^{156},e^{347}-e^{167},e^{357}}\\
(0, 0, 0, 12, 13, 24,0) & e^{123},e^{124},e^{125},e^{126},e^{127},e^{134},e^{135},e^{137},e^{145},e^{146},e^{147},e^{157},\\
\multicolumn{2}{|c|}{e^{234},e^{235},e^{236},e^{237},e^{245}-e^{136},e^{246},e^{247},e^{267},e^{256}+e^{346},e^{257}+e^{347},e^{357}}\\
(0, 0, 0, 12, 14, 15 + 24,0) & e^{123},e^{124},e^{125},e^{126},e^{127},e^{134},e^{135},e^{137},e^{145},e^{146},e^{147},e^{157},\\
\multicolumn{2}{|c|}{e^{234},e^{136}+e^{235},e^{237},e^{245},e^{246}-e^{156},e^{247},e^{257}+e^{167},e^{236}+e^{345},e^{457}-e^{267}}\\
(0, 0, 0, 12, 14, 15+ 23+ 24,0) & e^{123},e^{124},e^{125},e^{126},e^{127},e^{134},e^{135},e^{137},e^{145},e^{136}+e^{146},e^{147},\\
\multicolumn{2}{|c|}{e^{157},e^{234},e^{136}+e^{235},e^{237},e^{245},e^{236}+e^{246}-e^{156},e^{247},e^{236}+e^{345},e^{347}-e^{257}-e^{167},e^{457}-e^{267}}\\
(0, 0, 12, 13, 14, 15,0) & e^{123},e^{124},e^{125},e^{126},e^{127},e^{134},e^{135},e^{136},e^{137},e^{145},e^{146},e^{147},\\
\multicolumn{2}{|c|}{e^{156},e^{157},e^{167},e^{234},e^{237},e^{245}-e^{236},e^{347}-e^{257}}\\
(0,0,12,13,14,23+15,0) & e^{123},e^{124},e^{125},e^{126},e^{127},e^{134},e^{135},e^{136},e^{137},e^{145},e^{147},e^{157},\\
\multicolumn{2}{|c|}{e^{234},e^{235}-e^{146},e^{236}-e^{156},e^{237},e^{245}-e^{156},e^{247}+e^{167},e^{347}-e^{257}}\\
(0,0,0,12,13+42,14+23,0) & e^{123},e^{124},e^{125},e^{126},e^{127},e^{134},e^{136},e^{137},e^{145},-e^{135}+e^{146},e^{147},\\
\multicolumn{2}{|c|}{e^{234},e^{235},e^{236}-e^{135},e^{237},e^{245}+e^{135},e^{246},e^{247},e^{167}+e^{257},-e^{157}+e^{267},-e^{167}+e^{347}}\\
(0,0,0,12,14,13+42,0) & e^{123},e^{124},e^{125},e^{126},e^{127},e^{134},e^{135},e^{137},e^{145},e^{146},e^{147},e^{157},\\
\multicolumn{2}{|c|}{e^{234},e^{235}-e^{136},e^{236},e^{237},e^{245},e^{246}+e^{136},e^{247},e^{257}-e^{167},e^{267}+e^{347}}\\
(0,0,0,12,13+14,24,0) & e^{123},e^{124},e^{125},e^{126},e^{127},e^{134},e^{135},e^{137},e^{145},e^{146},e^{147},e^{157},\\
\multicolumn{2}{|c|}{e^{234},e^{235}+e^{136},e^{236},e^{237},e^{245}-e^{136},e^{246},e^{247},e^{267},e^{257}+e^{167}+e^{347}}\\
(0,0,0,12,13,14,0) & e^{123},e^{124},e^{125},e^{126},e^{127},e^{134},e^{135},e^{136},e^{137},e^{145},e^{146},e^{147},\\
\multicolumn{2}{|c|}{e^{156},e^{157},e^{167},e^{234},e^{235},e^{237},e^{245}+e^{236},e^{246},e^{247},e^{347}+e^{257},e^{357}}\\
(0,0,0,0,0,12,0) & e^{123},e^{124},e^{125},e^{126},e^{127},e^{134},e^{135},e^{136},e^{137},e^{145},e^{146},e^{147},\\
\multicolumn{2}{|c|}{e^{156},e^{157},e^{167},e^{234},e^{235},e^{236},e^{237},e^{245},e^{246},e^{247},e^{256},e^{257},e^{267},e^{345},e^{347},e^{357},e^{457}}\\
(0, 0, 0, 0, 12, 14 + 25,0) & e^{123},e^{124},e^{125},e^{126},e^{127},e^{134},e^{135},e^{137},e^{145},e^{147},e^{156},e^{157},\\
\multicolumn{2}{|c|}{e^{234},e^{235},e^{237},e^{245},e^{246},e^{247},-e^{146}+e^{256},e^{257},-e^{236}+e^{345},e^{347},e^{457}+e^{267}}\\
(0,0,0,0,12,34,0) &e^{123},e^{124},e^{125},e^{127},e^{134},e^{135},e^{136},e^{137},e^{145},e^{146},e^{147},e^{157},\\
\multicolumn{2}{|c|}{e^{234},e^{235},e^{236},e^{237},e^{245},e^{246},e^{247},e^{257},e^{345}-e^{126},e^{346},e^{347},e^{367},e^{467}}\\
(0,0,0,0,12,15,0) & e^{123},e^{124},e^{125},e^{126},e^{127},e^{134},e^{135},e^{136},e^{137},e^{145},e^{146},e^{147},\\
\multicolumn{2}{|c|}{e^{156},e^{157},e^{167},e^{234},e^{235},e^{237},e^{245},e^{247},e^{256},e^{257},e^{347}}\\
\hline
\end{array}\]
\end{table}

\section{Indecomposable case}
\label{sec:classification}
In this section we complete the classification of $7$-dimensional nilpotent Lie algebras with a calibrated $\Gtwo$-structure. We have seen that there are  exactly three decomposable Lie algebras of this type. In order to discuss the indecomposable Lie algebras, we refer to Gong's classification in~\cite{Gong}. This list consists of 140 Lie algebras and 9 one-parameter families.

The one-parameter families are the following:
\begin{align*}
147E&=\left(0,0,0,e^{12},e^{23},- e^{13}, \lambda e^{26}-e^{15}- {(-1+\lambda)} e^{34}\right),& \lambda\neq0,1;\\
1357M&=\left(0,0,e^{12},0,e^{24}+e^{13},e^{14},- {(-1+\lambda)} e^{34}+e^{15}+ e^{26} \lambda\right),& \lambda\neq0;\\
1357N&=\left(0,0,e^{12},0,e^{13}+e^{24},e^{14},e^{46}+e^{34}+e^{15}+ e^{23} \lambda\right);\\
1357S&=\left(0,0,e^{12},0,e^{13},e^{24}+e^{23},e^{25}+e^{34}+e^{16}+e^{15}+ \lambda e^{26}\right), &\lambda\neq1;\\
12457N&=\left(0,0,e^{12},e^{13},e^{23},e^{24}+e^{15}, \lambda e^{25}+e^{26}+e^{34}-e^{35}+e^{16}+e^{14}\right);\\
123457I&=\left(0,0,e^{12},e^{13},e^{14}+e^{23},e^{15}+e^{24}, \lambda e^{25}- {(-1+\lambda)} e^{34}+e^{16}\right);\\
147E1&=\left(0,0,0,e^{12},e^{23},- e^{13},2 e^{26}-2 e^{34}- e^{16} \lambda+ \lambda e^{25}\right) , &\lambda>1;\\
1357QRS1&=\left(0,0,e^{12},0,e^{13}+e^{24},e^{14}-e^{23}, e^{26} \lambda+e^{15}- e^{34} {(-1+\lambda)}\right), &\lambda\neq0;\\
12457N2&=\left(0,0,e^{12},e^{13},e^{23},-e^{14}-e^{25},e^{15}-e^{35}+e^{16}+e^{24}+ e^{25} \lambda\right), & \lambda\geq0.
\end{align*}

Recall
that a $3$-form of type $\Gtwo$ has the form \eqref{eqn:Gtwoform}
with respect to some coframe $e^1,\dotsc, e^7$; such a coframe identifies the $\Gtwo$-structure.
\begin{lemma}
\label{lemma:withparameter}
Exactly three of the above Lie algebras admit a calibrated $\Gtwo$-structure. Explicit examples are given in terms of a coframe by
\begin{gather*}
1357N(\lambda=1): \quad
\sqrt3(2  e^{1}-  e^{7}-  e^{6}-  e^{5}),
\sqrt{3} (e^{4}+ e^{3}-2  e^{2}- e^{6}+  e^{5}), 2 e^{3}-e^{6}, 2  e^{5},\\
-e^{3}+3 e^{4}+e^{5}-e^{6},2 e^{3}-e^{5}-e^{6}+3 e^{7},- \sqrt{3} e^{6};\\
1357S(\lambda=-3):\quad \sqrt7(2e^1+e^2-e^5+e^6),7e^2+3e^5+5e^6,\sqrt7(e^3+2e^4-\frac{3}{2}e^7),\\3e^3+\frac{7}{2}e^7,
-\sqrt{70}e^6,\sqrt{10}(2e^5+e^6),-2\sqrt{10}e^3.\\
147E1(\lambda=2): \quad
\sqrt3(2  e^{1}+ e^{5}- e^{2}+  e^{6}),3 e^{2}-e^{5}+e^{6},e^{3}+2 e^{4}, \sqrt{3}( e^{3}+ e^{7}), \\
\sqrt{2} (e^{6}-e^5), \sqrt{6} (e^{5}+e^{6}),2  \sqrt{2} (e^{4}- e^{3}).
\end{gather*}
\end{lemma}
\begin{proof}
It is straightforward to verify that each coframe in the statement determines a calibrated $\Gtwo$-structure on the corresponding Lie algebra.
Conversely, for
each Lie algebra $\lie{g}$ the vector $e_7$ is in the center, and determines an epimorphism on a $6$-dimensional Lie algebra $\lie{h}$; we view $de^7$ as the curvature form on $\lie{h}$, and apply Proposition~\ref{prop:epi}.

In the case of $1357M$, the generic element of $H^2(\lie{h}^*)$ is represented by
\[\omega =   \lambda_6 e^{46}+ \lambda_3 e^{23}+ \lambda_1 e^{13}+ \lambda_5 {(e^{15}+e^{34})}+ \lambda_2 e^{16}
+ \lambda_4 {(e^{15}+e^{26})}.\]
Assume $de^7\wedge\omega$ is exact. Then $\lambda_3,\lambda_6$ are zero, $\lambda_4=-  \lambda_5 \lambda$ and
\[(\lambda-\lambda^2-1)\lambda_5=0.\]
Since $\lambda^2-\lambda+1$ has no real zeroes, $\lambda_4$ and $\lambda_5$ are zero as well, and therefore $\omega^3$ is zero. So there is no symplectic form in the cohomology class of $\omega$.
By Proposition~\ref{prop:epi}, if a calibrated $\Gtwo$-structure existed, then $\lie{g}$ would have to be decomposable, which is absurd.

The other cases are similar.
\end{proof}

We now turn to the rest of the list, where no parameters appear.
\begin{lemma}
\label{lemma:noparameters}
In Gong's list, only six Lie algebras with  no parameters in their definition admit a calibrated $\Gtwo$-structure, which  can be expressed in terms of a coframe as follows:
\begin{align*}
&0,0,12,0,0,13+24,15 && e^{1},e^{2},e^{5},e^{6},e^{3},e^{7},e^{4}\\
&0,0,12,0,0,13,14+25 && e^{1},e^{3},e^{5},e^{7},e^{2},e^{6},e^{4}\\
&0,0,0,12,13,14,15 && e^{1},e^{2},e^{4},e^{7},e^{5},e^{6},e^{3}\\
&0,0,0,12,13,14+23,15 && e^{2}+e^{7},e^{3}+e^{6},e^{7},e^{6},e^{5},e^{4},e^{1}\\
&0,0,12,13,23,15+24,16+34 && e^{2}+e^{4},e^{7},e^{2},e^{5},e^{3},e^{6},e^{1}\\
&0,0,12,13,23,15+24,16+25+34&& \sqrt{3} {(2 e^{2}+e^{5}+e^{7})},2 e^{4}-3 e^{5}-e^{7}, \sqrt{3} {(e^{1}-e^{3}+2 e^{6})},\\ &&&e^{1}+3 e^{3}, \sqrt{6} e^{7}, \sqrt{2} {(2 e^{4}-e^{7})},2 \sqrt{2}  e^{1}\\
\end{align*}
\end{lemma}

\begin{theorem}
Up to isomorphism, 
there are exactly $12$ nilpotent Lie algebras that admit a calibrated $\Gtwo$-structure, namely those appearing in Theorem~\ref{G2-calibrated-decomposable}, Lemma~\ref{lemma:withparameter} and Lemma~\ref{lemma:noparameters}.
\end{theorem}
\begin{proof}
We must show that the remaining Lie algebras in Gong's list satisfy one of the two obstructions of Section~\ref{sec:obstructions}; we do so in the Appendix, where we reproduce Gong's list, and note which obstruction applies to each Lie algebra (as a preference, we try to use Proposition~\ref{prop:epi} rather than Lemma~\ref{lemma:obstr2} whenever possible, because the former does not require computing the space of closed $3$-forms).
\end{proof}

\section*{Appendix}
This appendix contains a list of all indecomposable nilpotent Lie algebras of dimension $7$, taken from \cite{Gong}, except the $9$ one-parameter families that we listed at the beginning of Section~\ref{sec:classification}. Alongside each Lie algebra $\lie{g}$, we give a chosen vector $\xi\in\lie{g}$ which satisfies the conditions of Proposition~\ref{prop:epi} (when marked with a (P)) or Lemma~\ref{lemma:obstr2}, and the structure constants of the quotient $\lie{g}/\Span{\xi}$. The word ``resists'' marks instead the six Lie algebras that resist the obstructions. Below each Lie algebra, we give a basis of its space of closed $3$-forms, except when Proposition~\ref{prop:epi} applies.
\begin{table}
 \caption{\label{table:step2} Step 2 nilpotent Lie algebras of dimension $7$}
\[\begin{array}{|lcr|}
\hline
0,0,0,0,12,23,24 & e_{7}&[0,0,0,0,e^{12},e^{23}]\\
\multicolumn{3}{|c|}{e^{123},e^{124},e^{125},e^{126},e^{127},e^{134},e^{135},e^{136},e^{145},e^{137}+e^{146},e^{147},}\\
\multicolumn{3}{|c|}{e^{234},e^{235},e^{236},e^{237},e^{245},e^{246},e^{247},e^{256},e^{257},e^{267},e^{345}+e^{137},}\\
\multicolumn{3}{|c|}{e^{346},e^{347}}\\
\hline
0,0,0,0,12,23,34 & e_{7}&[0,0,0,0,e^{12},e^{23}]\\
\multicolumn{3}{|c|}{e^{123},e^{124},e^{125},e^{126},e^{134},e^{135},e^{136},e^{137},e^{145},e^{146}-e^{127},e^{147},}\\
\multicolumn{3}{|c|}{e^{234},e^{235},e^{236},e^{237},e^{245},e^{246},e^{247},e^{256},-e^{127}+e^{345},e^{346},e^{347},e^{367}}\\
\hline
0,0,0,0,12+34,23,24 & e_{5}&[0,0,0,0,e^{23},e^{24}]\\
\multicolumn{3}{|c|}{e^{123},e^{124},e^{126},e^{127},e^{134},e^{135},e^{136},e^{137}+e^{125},e^{145},-e^{125}+e^{146},}\\
\multicolumn{3}{|c|}{e^{147},e^{234},e^{235},e^{236},e^{237},e^{245},e^{246},e^{247},e^{267},-e^{125}+e^{345},e^{346},}\\
\multicolumn{3}{|c|}{e^{347},-e^{256}+e^{367},-e^{257}+e^{467}}\\
\hline
0,0,0,0,12+34,13,24 & e_{7}&[0,0,0,0,e^{12}+e^{34},e^{13}]\\
\multicolumn{3}{|c|}{e^{123},e^{124},e^{126},e^{127},e^{134},e^{135},e^{136},e^{137}+e^{125},e^{145},e^{146},e^{147},}\\
\multicolumn{3}{|c|}{e^{234},e^{235},e^{236},e^{237},e^{245},e^{246}+e^{125},e^{247},-e^{125}+e^{345},e^{346},e^{347}}\\
\hline
0,0,0,0,0,12,14+35 & e_{6}(P)&[0,0,0,0,0,e^{14}+e^{35}]\\
\hline
0,0,0,0,0,12+34,15+23 & e_{7}(P)&[0,0,0,0,0,e^{34}+e^{12}]\\
\hline
0,0,0,0,0,0,12+34+56 & e_{7}&[0,0,0,0,0,0]\\
\multicolumn{3}{|c|}{e^{123},e^{124},e^{125},e^{126},e^{134},e^{135},e^{136},e^{145},e^{146},e^{156},e^{234},e^{235},}\\
\multicolumn{3}{|c|}{e^{236},e^{245},e^{246},e^{256},e^{345},e^{346},e^{356},e^{456}}\\
\hline
0,0,0,0,12-34,13+24,14 & e_{6}&[0,0,0,0,e^{12}-e^{34},e^{14}]\\
\multicolumn{3}{|c|}{e^{123},e^{124},e^{126},e^{127},e^{134},e^{135},e^{136}-e^{125},e^{137},e^{145},e^{146},e^{147},}\\
\multicolumn{3}{|c|}{e^{234},e^{235},e^{236},e^{125}+e^{237},e^{245},-e^{125}+e^{246},e^{247},e^{125}+e^{345},e^{346},}\\
\multicolumn{3}{|c|}{e^{347},e^{457}-e^{167},e^{467}+e^{157}}\\
\hline
0,0,0,0,12-34,13+24,14-23 & e_{7}&[0,0,0,0,e^{12}-e^{34},e^{24}+e^{13}]\\
\multicolumn{3}{|c|}{e^{123},e^{124},e^{126},e^{127},e^{134},e^{135},e^{136}-e^{125},e^{137},e^{145},e^{146},-e^{125}+e^{147},}\\
\multicolumn{3}{|c|}{e^{234},e^{235},e^{236},e^{125}+e^{237},e^{245},-e^{125}+e^{246},e^{247},e^{125}+e^{345},e^{346},e^{347}}\\
\hline
\end{array}\]
\end{table}

\begin{table}
 \caption{\label{table:step3} Step 3 nilpotent Lie algebras of dimension $7$}
\[\begin{array}{|lcr|}
\hline
0,0,12,0,13,24,14 & e_{5}&[0,0,e^{12},0,e^{24},e^{14}]\\
\multicolumn{3}{|c|}{e^{123},e^{124},e^{125},e^{126},e^{127},e^{134},e^{135},e^{137},e^{145},e^{146},e^{147},e^{157},}\\
\multicolumn{3}{|c|}{e^{234},e^{235},e^{236},e^{237}+e^{136},e^{245}-e^{136},e^{246},e^{247},e^{345}-e^{156},e^{346}+e^{267},}\\
\multicolumn{3}{|c|}{e^{347}+e^{167},e^{467}}\\
\hline
0,0,12,0,13,23,14 & e_{7}(P)&[0,0,e^{12},0,e^{13},e^{23}]\\
\hline
0,0,12,0,13+24,23,14 & e_{6}&[0,0,e^{12},0,e^{24}+e^{13},e^{14}]\\
\multicolumn{3}{|c|}{e^{123},e^{124},e^{125},e^{126},e^{127},e^{134},e^{136},e^{137},e^{145},e^{146}+e^{135},e^{147},}\\
\multicolumn{3}{|c|}{e^{234},e^{235},e^{236},e^{237}+e^{135},-e^{135}+e^{245},e^{246},e^{247},e^{257}+e^{345}+e^{167},}\\
\multicolumn{3}{|c|}{e^{346}+e^{267},e^{347}+e^{157}}\\
\hline
0,0,12,0,0,13+24,15& \text{ resists }&\\
\multicolumn{3}{|c|}{e^{123},e^{124},e^{125},e^{126},e^{127},e^{134},e^{135},e^{137},e^{145},e^{146},e^{147},e^{157},}\\
\multicolumn{3}{|c|}{e^{234},e^{235},e^{236},e^{245},-e^{136}+e^{246},-e^{156}+e^{247},e^{256}+e^{237},e^{257},e^{345}-e^{156},}\\
\multicolumn{3}{|c|}{-e^{347}-e^{167}+e^{456},e^{457}}\\
\hline
0,0,12,0,0,13,14+25& \text{ resists }&\\
\multicolumn{3}{|c|}{e^{123},e^{124},e^{125},e^{126},e^{127},e^{134},e^{135},e^{136},e^{145},e^{146},e^{156},e^{157},}\\
\multicolumn{3}{|c|}{e^{234},e^{235},e^{236},e^{245},e^{237}+e^{246},e^{247},e^{256}-e^{137},e^{257}-e^{147},e^{345}+e^{147},}\\
\multicolumn{3}{|c|}{-e^{167}+e^{356},e^{457}}\\
\hline
0,0,12,0,0,13+24,25 & e_{7}&[0,0,e^{12},0,0,e^{13}+e^{24}]\\
\multicolumn{3}{|c|}{e^{123},e^{124},e^{125},e^{126},e^{127},e^{134},e^{135},e^{145},e^{146},e^{147}+e^{156},e^{157},}\\
\multicolumn{3}{|c|}{e^{234},e^{235},e^{236},e^{237},e^{245},-e^{136}+e^{246},e^{247},e^{256}-e^{137},e^{257},e^{345}+e^{147},e^{457}}\\
\hline
0,0,12,0,0,13+24,14+25 & e_{7}&[0,0,e^{12},0,0,e^{13}+e^{24}]\\
\multicolumn{3}{|c|}{e^{123},e^{124},e^{125},e^{126},e^{127},e^{134},e^{135},e^{145},e^{146},e^{156}+e^{147},e^{157},}\\
\multicolumn{3}{|c|}{e^{234},e^{235},e^{236},e^{237}+e^{136},e^{245},e^{246}-e^{136},e^{247},e^{256}-e^{137},e^{257}-e^{147},}\\
\multicolumn{3}{|c|}{e^{345}+e^{147},e^{457}}\\
\hline
0,0,12,0,0,13+45,24 & e_{7}(P)&[0,0,e^{12},0,0,e^{13}+e^{45}]\\
\hline
\hline
0,0,12,0,0,13+45,15+24 & e_{7}(P)&[0,0,e^{12},0,0,e^{13}+e^{45}]\\
\hline
0,0,12,0,0,13+24,45 & e_{7}&[0,0,e^{12},0,0,e^{13}+e^{24}]\\
\multicolumn{3}{|c|}{e^{123},e^{124},e^{125},e^{126},e^{134},e^{135},e^{145},e^{146},e^{147},-e^{127}+e^{156},e^{157},}\\
\multicolumn{3}{|c|}{e^{234},e^{235},e^{236},e^{245},e^{246}-e^{136},e^{247},e^{257},e^{345}-e^{127},e^{456}-e^{137},e^{457}}\\
\hline
0,0,12,0,0,13+14,15+23 & e_{7}&[0,0,e^{12},0,0,e^{14}+e^{13}]\\
\multicolumn{3}{|c|}{e^{123},e^{124},e^{125},e^{126},e^{127},e^{134},e^{135},e^{136},e^{137},e^{145},e^{146},e^{156},}\\
\multicolumn{3}{|c|}{e^{234},e^{235},-e^{147}+e^{236},e^{237}-e^{157},e^{245},e^{147}+e^{246},e^{247}+e^{256}+e^{157},}\\
\multicolumn{3}{|c|}{e^{257},-e^{247}+e^{345}}\\
\hline
0,0,12,0,0,13+24,15+23 & e_{7}&[0,0,e^{12},0,0,e^{24}+e^{13}]\\
\multicolumn{3}{|c|}{e^{123},e^{124},e^{125},e^{126},e^{127},e^{134},e^{135},e^{137},e^{145},e^{146},e^{147}+e^{136},}\\
\multicolumn{3}{|c|}{e^{234},e^{235},e^{236},-e^{157}+e^{237},e^{245},e^{246}-e^{136},-e^{156}+e^{247},e^{157}+e^{256},}\\
\multicolumn{3}{|c|}{e^{257},-e^{156}+e^{345}}\\
\hline
0,0,12,0,0,13,23+45 & e_{6}(P)&[0,0,e^{12},0,0,e^{23}+e^{45}]\\
\hline
0,0,12,0,0,13+24,23+45 & e_{6}(P)&[0,0,e^{12},0,0,e^{23}+e^{45}]\\
\hline
\end{array}\]\end{table}
\begin{table}
 \[\begin{array}{|lcr|}
\hline
0,0,0,12,13,14,15& \text{ resists }&\\
\multicolumn{3}{|c|}{e^{123},e^{124},e^{125},e^{126},e^{127},e^{134},e^{135},e^{136},e^{137},e^{145},e^{146},e^{147},}\\
\multicolumn{3}{|c|}{e^{156},e^{157},e^{167},e^{234},e^{235},e^{245}+e^{236},e^{246},e^{345}+e^{237},e^{346}+e^{256}+e^{247},}\\
\multicolumn{3}{|c|}{e^{257}+e^{347}+e^{356},e^{357}}\\
\hline
0,0,0,12,13,14,35 & e_{7}&[0,0,0,e^{12},e^{13},e^{14}]\\
\multicolumn{3}{|c|}{e^{123},e^{124},e^{125},e^{126},e^{134},e^{135},e^{136},e^{137},e^{145},e^{146},e^{156},e^{157},}\\
\multicolumn{3}{|c|}{e^{234},e^{235},e^{237},e^{245}+e^{236},e^{246},e^{345}+e^{127},e^{347}+e^{257},-e^{147}+e^{356},e^{357}}\\
\hline
0,0,0,12,13,14+35,15 & e_{7}(P)&[0,0,0,e^{12},e^{13},e^{35}+e^{14}]\\
\hline
0,0,0,12,13,14,25+34 & e_{6}(P)&[0,0,0,e^{12},e^{13},e^{34}+e^{25}]\\
\hline
0,0,0,12,13,14+15,25+34 & e_{6}(P)&[0,0,0,e^{12},e^{13},e^{25}+e^{34}]\\
\hline
0,0,0,12,13,24+35,25+34 & e_{7}(P)&[0,0,0,e^{12},e^{13},e^{24}+e^{35}]\\
\hline
0,0,0,12,13,14+15+24+35,25+34 & e_{7}(P)&[0,0,0,e^{12},e^{13},e^{35}+e^{24}+e^{15}+e^{14}]\\
\hline
0,0,0,12,13,14+24+35,25+34 & e_{7}(P)&[0,0,0,e^{12},e^{13},e^{24}+e^{35}+e^{14}]\\
\hline
0,0,0,12,13,25+34,35 & e_{7}(P)&[0,0,0,e^{12},e^{13},e^{25}+e^{34}]\\
\hline
0,0,0,12,13,15+35,25+34 & e_{6}(P)&[0,0,0,e^{12},e^{13},e^{34}+e^{25}]\\
\hline
0,0,0,12,13,14+35,25+34 & e_{7}(P)&[0,0,0,e^{12},e^{13},e^{14}+e^{35}]\\
\hline
0,0,0,12,13,14+23,15& \text{ resists }&\\
\multicolumn{3}{|c|}{e^{123},e^{124},e^{125},e^{126},e^{127},e^{134},e^{135},e^{136},e^{137},e^{145},e^{147},e^{157},}\\
\multicolumn{3}{|c|}{e^{234},e^{235},-e^{146}+e^{236},-e^{156}+e^{237},e^{245}+e^{146},e^{246},e^{167}+e^{257},e^{156}+e^{345},}\\
\multicolumn{3}{|c|}{e^{346}+e^{256}+e^{247},-e^{167}+e^{356}+e^{347},e^{357}}\\
\hline
0,0,0,12,13,14+23,35 & e_{7}&[0,0,0,e^{12},e^{13},e^{23}+e^{14}]\\
\multicolumn{3}{|c|}{e^{123},e^{124},e^{125},e^{126},e^{134},e^{135},e^{136},e^{137},e^{145},-e^{127}+e^{156},e^{157},}\\
\multicolumn{3}{|c|}{e^{234},e^{235},-e^{146}+e^{236},e^{237},e^{146}+e^{245},e^{246},e^{127}+e^{345},e^{257}+e^{347},}\\
\multicolumn{3}{|c|}{-e^{147}+e^{356},e^{357}}\\
\hline
0,0,0,12,13,15+24,23 & e_{7}(P)&[0,0,0,e^{12},e^{13},e^{24}+e^{15}]\\
\hline
0,0,0,12,13,14+35,15+23 & e_{7}(P)&[0,0,0,e^{12},e^{13},e^{35}+e^{14}]\\
\hline
0,0,0,12,13,23,25+34 & e_{6}(P)&[0,0,0,e^{12},e^{13},e^{25}+e^{34}]\\
\hline
0,0,0,12,13,14+23,25+34 & e_{6}(P)&[0,0,0,e^{12},e^{13},e^{25}+e^{34}]\\
\hline
0,0,0,12,13,14+15+23,25+34 & e_{6}(P)&[0,0,0,e^{12},e^{13},e^{25}+e^{34}]\\
\hline
0,0,12,0,0,0,13+24+56 & e_{7}&[0,0,e^{12},0,0,0]\\
\multicolumn{3}{|c|}{e^{123},e^{124},e^{125},e^{126},e^{134},e^{135},e^{136},e^{145},e^{146},e^{156},e^{234},e^{235},}\\
\multicolumn{3}{|c|}{e^{236},e^{245},e^{246},e^{256},-e^{157}+e^{345},-e^{167}+e^{346},e^{356}-e^{127},e^{456}}\\
\hline
0,0,0,12,13,0,16+25+34 & e_{7}&[0,0,0,e^{12},e^{13},0]\\
\multicolumn{3}{|c|}{e^{123},e^{124},e^{125},e^{126},e^{134},e^{135},e^{136},e^{145},e^{146},e^{156},e^{234},e^{235},}\\
\multicolumn{3}{|c|}{e^{236},e^{245}+e^{127},e^{246},e^{256}-e^{237},e^{345}-e^{137},e^{346}+e^{237},e^{356}}\\
\hline
0,0,0,12,13,0,14+26+35 & e_{7}&[0,0,0,e^{12},e^{13},0]\\
\multicolumn{3}{|c|}{e^{123},e^{124},e^{125},e^{126},e^{134},e^{135},e^{136},e^{145},e^{146},e^{156},e^{234},e^{235},}\\
\multicolumn{3}{|c|}{e^{236},e^{237}+e^{245},e^{246},e^{137}+e^{256},e^{127}+e^{345},-e^{137}+e^{346},e^{356},e^{157}+e^{456}-e^{367}}\\
\hline
0,0,0,12,23,-13,15+26+16-2*34 & e_{7}&[0,0,0,e^{12},e^{23},- e^{13}]\\
\multicolumn{3}{|c|}{e^{123},e^{124},e^{125},e^{126},e^{134},e^{135},e^{136},\frac{1}{2} e^{127}+e^{145},e^{146},}\\
\multicolumn{3}{|c|}{e^{156}-e^{137},e^{234},e^{235},e^{236},e^{245},e^{246}+\frac{1}{2} e^{127},e^{256}-e^{137}-e^{237},}\\
\multicolumn{3}{|c|}{e^{345}+e^{137}+e^{237},e^{346}-e^{137},e^{356}}\\
\hline
0,0,0,0,12,34,15+36 & e_{7}&[0,0,0,0,e^{12},e^{34}]\\
\multicolumn{3}{|c|}{e^{123},e^{124},e^{125},e^{134},e^{135},e^{136},e^{137},e^{145},e^{146},e^{234},e^{235},e^{236},}\\
\multicolumn{3}{|c|}{e^{245},e^{246},e^{345}-e^{126},e^{346},e^{347}-e^{156},e^{356}+e^{127},-e^{157}+e^{367}}\\
\hline
\end{array}\]\end{table}
\begin{table}
\[\begin{array}{|lcr|}
\hline
0,0,0,0,12,34,15+24+36 & e_{7}&[0,0,0,0,e^{12},e^{34}]\\
\multicolumn{3}{|c|}{e^{123},e^{124},e^{125},e^{134},e^{135},e^{136},e^{137}+e^{126},e^{145},e^{146},e^{234},e^{235},}\\
\multicolumn{3}{|c|}{e^{236},e^{245},e^{246},e^{345}-e^{126},e^{346},e^{347}-e^{156},e^{356}+e^{127}}\\
\hline
0,0,0,0,12,14+23,16-35 & e_{7}&[0,0,0,0,e^{12},e^{14}+e^{23}]\\
\multicolumn{3}{|c|}{e^{123},e^{124},e^{125},e^{126},e^{134},e^{135},e^{136},e^{137},e^{145},e^{156}+e^{127},e^{157},}\\
\multicolumn{3}{|c|}{e^{234},e^{235},-e^{146}+e^{236},e^{245},e^{246},-e^{146}+e^{345},e^{346},e^{147}+e^{237}+e^{356}}\\
\hline
0,0,0,0,12,14+23,16+24-35 & e_{7}&[0,0,0,0,e^{12},e^{23}+e^{14}]\\
\multicolumn{3}{|c|}{e^{123},e^{124},e^{125},e^{126},e^{134},e^{135},e^{136},e^{145},e^{137}+e^{146},e^{156}+e^{127},}\\
\multicolumn{3}{|c|}{e^{234},e^{235},e^{236}+e^{137},e^{245},e^{246},e^{157}+e^{256},e^{345}+e^{137},e^{346},e^{356}+e^{237}+e^{147}}\\
\hline
0,0,12,0,0,13+14+25,15+23 & e_{7}&[0,0,e^{12},0,0,e^{25}+e^{14}+e^{13}]\\
\multicolumn{3}{|c|}{e^{123},e^{124},e^{125},e^{126},e^{127},e^{134},e^{135},e^{137},e^{145},e^{156},e^{157}+e^{136},}\\
\multicolumn{3}{|c|}{e^{234},e^{235},-e^{147}+e^{236},e^{136}+e^{237},e^{245},e^{147}+e^{246},e^{146}+e^{247},e^{256}-e^{146}-e^{136},}\\
\multicolumn{3}{|c|}{e^{257},e^{345}+e^{146}}\\
\hline
0,0,0,12,13,14,24+35 & e_{6}(P)&[0,0,0,e^{12},e^{13},e^{24}+e^{35}]\\
\hline
0,0,0,12,13,24-35,25+34 & e_{7}(P)&[0,0,0,e^{12},e^{13},-e^{35}+e^{24}]\\
\hline
0,0,0,12,13,14+24-35,25+34 & e_{7}(P)&[0,0,0,e^{12},e^{13},e^{24}-e^{35}+e^{14}]\\
\hline
0,0,0,12,13,23,24+35 & e_{6}(P)&[0,0,0,e^{12},e^{13},e^{35}+e^{24}]\\
\hline
0,0,0,12,13,14+23,24+35 & e_{6}(P)&[0,0,0,e^{12},e^{13},e^{24}+e^{35}]\\
\hline
0,0,0,12,13,0,16+24+35 & e_{7}&[0,0,0,e^{12},e^{13},0]\\
\multicolumn{3}{|c|}{e^{123},e^{124},e^{125},e^{126},e^{134},e^{135},e^{136},e^{145},e^{146},e^{156},e^{234},e^{235},}\\
\multicolumn{3}{|c|}{e^{236},e^{245}-e^{137},e^{246},e^{256}-e^{237},e^{127}+e^{345},e^{237}+e^{346},e^{356}}\\
\hline
0,0,0,0,13+24,14-23,15+26 & e_{7}&[0,0,0,0,e^{13}+e^{24},e^{14}-e^{23}]\\
\multicolumn{3}{|c|}{e^{123},e^{124},e^{125},e^{126},e^{127},e^{134},e^{136},e^{145},-e^{135}+e^{146},e^{234},e^{235},}\\
\multicolumn{3}{|c|}{e^{135}+e^{236},e^{237}-e^{147}+e^{156},-e^{135}+e^{245},e^{246},e^{137}+e^{247}+e^{256},e^{257}-e^{167},}\\
\multicolumn{3}{|c|}{e^{345},e^{346}}\\
\hline
0,0,0,0,13+24,14-23,15+26+24 & e_{7}&[0,0,0,0,e^{24}+e^{13},e^{14}-e^{23}]\\
\multicolumn{3}{|c|}{e^{123},e^{124},e^{125},e^{126},e^{127},e^{134},e^{136},e^{145},e^{146}-e^{135},e^{234},e^{235},}\\
\multicolumn{3}{|c|}{e^{135}+e^{236},e^{156}-e^{147}+e^{237},-e^{135}+e^{245},e^{246},e^{256}+e^{137}-e^{135}+e^{247},}\\
\multicolumn{3}{|c|}{e^{257}-e^{167}-e^{147},e^{345},e^{346}}\\
\hline
\end{array}\]
\end{table}

\begin{table}
  \caption{\label{table:step4} Step 4 nilpotent Lie algebras of dimension $7$}
\[\begin{array}{|lcr|}
\hline
0,0,12,13,0,14,15 & e_{7}&[0,0,e^{12},e^{13},0,e^{14}]\\
\multicolumn{3}{|c|}{e^{123},e^{124},e^{125},e^{126},e^{127},e^{134},e^{135},e^{136},e^{137},e^{145},e^{146},e^{147},}\\
\multicolumn{3}{|c|}{e^{156},e^{157},e^{167},e^{234},e^{235},-e^{237}+e^{245},e^{257},e^{256}+e^{345}}\\
\hline
0,0,12,13,0,25,14 & e_{7}&[0,0,e^{12},e^{13},0,e^{25}]\\
\multicolumn{3}{|c|}{e^{123},e^{124},e^{125},e^{126},e^{127},e^{134},e^{135},e^{137},e^{145},e^{147},e^{156},e^{157},}\\
\multicolumn{3}{|c|}{e^{234},e^{235},e^{236},e^{136}+e^{245},e^{256},-e^{146}+e^{257},e^{146}+e^{345}}\\
\hline
0,0,12,13,0,14+25,15 & e_{7}&[0,0,e^{12},e^{13},0,e^{25}+e^{14}]\\
\multicolumn{3}{|c|}{e^{123},e^{124},e^{125},e^{126},e^{127},e^{134},e^{135},e^{137},e^{145},e^{147},e^{156},e^{157},}\\
\multicolumn{3}{|c|}{e^{234},e^{235},e^{136}+e^{237},e^{136}+e^{245},e^{256}-e^{146},e^{257},e^{345}+e^{146},e^{357}+e^{167}}\\
\hline
\end{array}\]
\end{table}

\begin{table}
\[\begin{array}{|lcr|}
\hline
0,0,12,13,0,14+23+25,15 & e_{6}&[0,0,e^{12},e^{13},0,e^{15}]\\
\multicolumn{3}{|c|}{e^{123},e^{124},e^{125},e^{126},e^{127},e^{134},e^{135},e^{137},e^{145},}\\
\multicolumn{3}{|c|}{e^{147},e^{136}+e^{156},e^{157},e^{234},e^{235},e^{237}+e^{136},e^{245}+e^{136},}\\
\multicolumn{3}{|c|}{e^{256}+e^{236}-e^{146},e^{257},e^{345}-e^{236}+e^{146},e^{247}-e^{236}+e^{357}+e^{167}+e^{146}}\\
\hline
0,0,12,13,0,23+25,14 & e_{7}&[0,0,e^{12},e^{13},0,e^{23}+e^{25}]\\
\multicolumn{3}{|c|}{e^{123},e^{124},e^{125},e^{126},e^{127},e^{134},e^{135},e^{137},e^{145},e^{147},e^{156}+e^{136},}\\
\multicolumn{3}{|c|}{e^{157},e^{234},e^{235},e^{236},e^{136}+e^{245},e^{256},e^{257}-e^{146}+e^{237},e^{345}+e^{146}-e^{237}}\\
\hline
0,0,12,13,0,14+23,15 & e_{7}&[0,0,e^{12},e^{13},0,e^{23}+e^{14}]\\
\multicolumn{3}{|c|}{e^{123},e^{124},e^{125},e^{126},e^{127},e^{134},e^{135},e^{136},e^{137},e^{145},e^{147},e^{157},e^{234},e^{235},}\\
\multicolumn{3}{|c|}{-e^{146}+e^{236},-e^{156}+e^{237},e^{245}-e^{156},e^{247}+e^{256}+e^{167},e^{257},e^{345}-e^{247}-e^{167}}\\
\hline
0,0,12,13,0,15+23,14 & e_{7}&[0,0,e^{12},e^{13},0,e^{15}+e^{23}]\\
\multicolumn{3}{|c|}{e^{123},e^{124},e^{125},e^{126},e^{127},e^{134},e^{135},e^{136},e^{137},e^{145},e^{147},e^{157},}\\
\multicolumn{3}{|c|}{e^{234},e^{235},-e^{156}+e^{236},-e^{146}+e^{237},-e^{156}+e^{245},e^{167}+e^{247},e^{256},e^{345}+e^{257}}\\
\hline
0,0,12,13,0,23,14+25 & e_{7}(P)&[0,0,e^{12},e^{13},0,e^{23}]\\
\hline
0,0,12,13,0,14+23,25 & e_{7}&[0,0,e^{12},e^{13},0,e^{14}+e^{23}]\\
\multicolumn{3}{|c|}{e^{123},e^{124},e^{125},e^{126},e^{127},e^{134},e^{135},e^{136},e^{145},e^{137}+e^{156},e^{157},}\\
\multicolumn{3}{|c|}{e^{234},e^{235},-e^{146}+e^{236},e^{237},e^{137}+e^{245},-e^{147}+e^{256},e^{257},e^{147}+e^{345}}\\
\hline
0,0,12,13,0,14+23,23+25 & e_{7}&[0,0,e^{12},e^{13},0,e^{23}+e^{14}]\\
\multicolumn{3}{|c|}{e^{123},e^{124},e^{125},e^{126},e^{127},e^{134},e^{135},e^{136},e^{145},e^{156}+e^{137},e^{157}+e^{137},}\\
\multicolumn{3}{|c|}{e^{234},e^{235},-e^{146}+e^{236},e^{237},e^{137}+e^{245},e^{146}-e^{147}+e^{256},e^{257},-e^{146}+e^{147}+e^{345}}\\
\hline
0,0,12,13,0,15+23,14+25 & e_{7}&[0,0,e^{12},e^{13},0,e^{15}+e^{23}]\\
\multicolumn{3}{|c|}{e^{123},e^{124},e^{125},e^{126},e^{127},e^{134},e^{135},e^{136},e^{145},e^{156}+e^{137},e^{157},e^{234},e^{235},e^{236}+e^{137},}\\
\multicolumn{3}{|c|}{-e^{146}+e^{237},e^{245}+e^{137},e^{256},-e^{147}+e^{257},e^{345}+e^{147},e^{356}-e^{167}-e^{247}}\\
\hline
0,0,12,13,23,14+25,15+24 & e_{7}(P)&[0,0,e^{12},e^{13},e^{23},e^{14}+e^{25}]\\
\hline
0,0,12,13,23,24+15,14 & e_{6}(P)&[0,0,e^{12},e^{13},e^{23},e^{14}]\\
\hline
0,0,0,12,14+23,23,15-34 & e_{7}&[0,0,0,e^{12},e^{23}+e^{14},e^{23}]\\
\multicolumn{3}{|c|}{e^{123},e^{124},e^{125},e^{126},e^{134},e^{135},e^{136},e^{137},e^{145}+e^{127},e^{146}+e^{127},e^{147},e^{234},}\\
\multicolumn{3}{|c|}{e^{235}+e^{127},e^{236},e^{245},e^{246},e^{345}+e^{237},-e^{156}+e^{346}+e^{237},e^{167}-\frac{1}{2} e^{157}-\frac{1}{2} e^{347}+e^{356}}\\
\hline
0,0,0,12,14+23,13,15-34 & e_{7}&[0,0,0,e^{12},e^{23}+e^{14},e^{13}]\\
\multicolumn{3}{|c|}{e^{123},e^{124},e^{125},e^{126},e^{134},e^{135},e^{136},e^{137},e^{127}+e^{145},e^{146},e^{147},}\\
\multicolumn{3}{|c|}{e^{234},e^{127}+e^{235},e^{236},e^{245},-e^{127}+e^{246},e^{237}+e^{345},-e^{156}+e^{346},e^{167}+e^{356}}\\
\hline
0,0,0,12,14+23,24,15-34 & e_{7}&[0,0,0,e^{12},e^{14}+e^{23},e^{24}]\\
\multicolumn{3}{|c|}{e^{123},e^{124},e^{125},e^{126},e^{134},e^{135},e^{136}-e^{127},e^{137},e^{145}+e^{127},e^{146},}\\
\multicolumn{3}{|c|}{e^{147},e^{234},e^{235}+e^{127},e^{236},e^{245},e^{246},e^{345}+e^{237},e^{247}+e^{346}-e^{156}}\\
\hline
0,0,0,12,14+23,13+24,15-34 & e_{7}&[0,0,0,e^{12},e^{23}+e^{14},e^{24}+e^{13}]\\
\multicolumn{3}{|c|}{e^{123},e^{124},e^{125},e^{126},e^{134},e^{135},e^{136}-e^{127},e^{137},e^{145}+e^{127},e^{146},}\\
\multicolumn{3}{|c|}{e^{147},e^{234},e^{127}+e^{235},e^{236},e^{245},e^{246}-e^{127},e^{345}+e^{237},e^{346}-e^{156}+e^{247}}\\
\hline
0,0,12,13,0,0,14+56 & e_{7}&[0,0,e^{12},e^{13},0,0]\\
\multicolumn{3}{|c|}{e^{123},e^{124},e^{125},e^{126},e^{134},e^{135},e^{136},e^{145},e^{146},e^{156},e^{157},e^{167},}\\
\multicolumn{3}{|c|}{e^{234},e^{235},e^{236},e^{256},e^{257}+e^{345},e^{346}+e^{267},-e^{127}+e^{356},e^{456}-e^{137},e^{567}-e^{147}}\\
\hline
0,0,12,13,0,0,23+14+56 & e_{7}&[0,0,e^{12},e^{13},0,0]\\
\multicolumn{3}{|c|}{e^{123},e^{124},e^{125},e^{126},e^{134},e^{135},e^{136},e^{145},e^{146},e^{156},e^{234},e^{235},e^{236},-e^{157}+e^{245},}\\
\multicolumn{3}{|c|}{-e^{167}+e^{246},e^{256},e^{257}+e^{345},e^{346}+e^{267},-e^{127}+e^{356},-e^{137}+e^{456}}\\
\hline
0,0,0,12,14+23,0,15+26-34 & e_{7}&[0,0,0,e^{12},e^{23}+e^{14},0]\\
\multicolumn{3}{|c|}{e^{123},e^{124},e^{125},e^{126},e^{134},e^{135},e^{136},e^{145}+e^{127},e^{146},-e^{137}+e^{156},}\\
\multicolumn{3}{|c|}{e^{234},e^{235}+e^{127},e^{236},e^{245},e^{246},e^{256}+e^{147},e^{345}+e^{237},e^{346}-e^{137},e^{356}+e^{167}}\\
\hline
\end{array}\]
\end{table}

\begin{table}
\[\begin{array}{|lcr|}
\hline
0,0,0,12,14+23,0,15+36-34 & e_{7}&[0,0,0,e^{12},e^{23}+e^{14},0]\\
\multicolumn{3}{|c|}{e^{123},e^{124},e^{125},e^{126},e^{134},e^{135},e^{136},e^{137},e^{146},e^{127}+e^{156}+e^{145},-e^{147}+e^{167},}\\
\multicolumn{3}{|c|}{e^{234},e^{235}-e^{145},e^{236},e^{245},e^{246},e^{237}+e^{345},e^{346}+e^{127}+e^{145},e^{147}+e^{356}}\\
\hline
0,0,0,12,14+23,0,15+24+36-34 & e_{7}&[0,0,0,e^{12},e^{14}+e^{23},0]\\
\multicolumn{3}{|c|}{e^{123},e^{124},e^{125},e^{126},e^{134},e^{135},e^{136},e^{145}+e^{137},e^{146},e^{156}+e^{127}-e^{137},}\\
\multicolumn{3}{|c|}{e^{234},e^{235}+e^{137},e^{236},e^{245},e^{246},-e^{167}+e^{147}+e^{256},}\\
\multicolumn{3}{|c|}{e^{345}+e^{237},e^{346}+e^{127}-e^{137},e^{147}+e^{356}}\\
\hline
0,0,12,0,23,24,16+25+34 & e_{7}&[0,0,e^{12},0,e^{23},e^{24}]\\
\multicolumn{3}{|c|}{e^{123},e^{124},e^{125},e^{126},e^{134},e^{135},e^{136}+e^{127},e^{145}-e^{127},e^{146},e^{234},e^{235},e^{236},}\\
\multicolumn{3}{|c|}{e^{147}+e^{237}+e^{156},e^{245},e^{246},e^{256},e^{147}+e^{345},-e^{247}+e^{346},e^{456}+e^{267}}\\
\hline
0,0,12,0,23,24,25+46 & e_{7}&[0,0,e^{12},0,e^{23},e^{24}]\\
\multicolumn{3}{|c|}{e^{123},e^{124},e^{125},e^{126},e^{134},e^{135},e^{145}+e^{136},e^{146},e^{234},e^{235},e^{236},}\\
\multicolumn{3}{|c|}{e^{245},e^{246},e^{247},e^{256},e^{267},e^{147}+e^{345},-e^{127}+e^{346},e^{237}+e^{456},e^{467}-e^{257}}\\
\hline
0,0,12,0,23,24,13+25-46 & e_{7}&[0,0,e^{12},0,e^{23},e^{24}]\\
\multicolumn{3}{|c|}{e^{123},e^{124},e^{125},e^{126},e^{134},e^{135},e^{136}+e^{145},e^{146},e^{234},e^{235},e^{236},e^{245},}\\
\multicolumn{3}{|c|}{e^{246},e^{247}-e^{136},e^{256},-e^{147}+e^{267}-e^{156},e^{147}+e^{345},e^{127}+e^{346},-e^{237}+e^{456}}\\
\hline
0,0,12,0,23,14,16+25 & e_{7}&[0,0,e^{12},0,e^{23},e^{14}]\\
\multicolumn{3}{|c|}{e^{123},e^{124},e^{125},e^{126},e^{127},e^{134},e^{135},e^{136},e^{146},e^{234},e^{235},e^{236}-e^{145},}\\
\multicolumn{3}{|c|}{e^{156}+e^{237},e^{245},e^{246},-e^{147}+e^{256},-e^{167}+e^{257},e^{147}+e^{345},e^{346}-e^{247}}\\
\hline
0,0,12,0,23,14,16+25+26-34 & e_{7}&[0,0,e^{12},0,e^{23},e^{14}]\\
\multicolumn{3}{|c|}{e^{123},e^{124},e^{125},e^{126},e^{134},e^{135},e^{136},e^{127}+e^{145},e^{146},e^{234},e^{235},e^{127}+e^{236},}\\
\multicolumn{3}{|c|}{e^{156}+e^{237},e^{245},e^{246},e^{256}-e^{247}-e^{147},e^{345}+e^{247}+e^{147},e^{346}-e^{247}}\\
\hline
0,0,12,0,23,14,25+46 & e_{7}&[0,0,e^{12},0,e^{23},e^{14}]\\
\multicolumn{3}{|c|}{e^{123},e^{124},e^{125},e^{126},e^{134},e^{135},e^{136},e^{146},e^{234},e^{235},e^{236}-e^{145},}\\
\multicolumn{3}{|c|}{e^{245},e^{246},e^{247},e^{256}-e^{147},e^{345}+e^{147},e^{346}-e^{127},e^{237}+e^{456},-e^{257}+e^{467}}\\
\hline
0,0,12,0,23,14,13+25+46 & e_{7}&[0,0,e^{12},0,e^{23},e^{14}]\\
\multicolumn{3}{|c|}{e^{123},e^{124},e^{125},e^{126},e^{134},e^{135},e^{136},e^{146},e^{234},e^{235},e^{236}-e^{145},}\\
\multicolumn{3}{|c|}{e^{245},e^{246},e^{145}+e^{247},e^{256}-e^{147},e^{147}+e^{345},-e^{127}+e^{346},e^{237}+e^{456}}\\
\hline
0,0,12,0,13+24,14,15+23+1/2*(26+34) & e_{7}&[0,0,e^{12},0,e^{13}+e^{24},e^{14}]\\
\multicolumn{3}{|c|}{e^{123},e^{124},e^{125},e^{126},e^{134},2 e^{127}+e^{135},e^{136},e^{145},e^{146},4 e^{127}-2 e^{147}+e^{156},}\\
\multicolumn{3}{|c|}{e^{234},e^{235},-2 e^{127}+e^{236},2 e^{127}+e^{245},e^{246},2 e^{137}+e^{247}+e^{256},-e^{247}+e^{345},}\\
\multicolumn{3}{|c|}{-4 e^{127}+e^{346}+2 e^{147},e^{456}+4 e^{137}+2 e^{167}}\\
\hline
0,0,12,0,13+24,23,16+25 & e_{7}&[0,0,e^{12},0,e^{24}+e^{13},e^{23}]\\
\multicolumn{3}{|c|}{e^{123},e^{124},e^{125},e^{126},e^{127},e^{134},e^{136},e^{145},e^{146}+e^{135},e^{234},e^{235},e^{236},-e^{135}+e^{245},}\\
\multicolumn{3}{|c|}{e^{246},e^{156}+e^{137}+e^{247},e^{256}-e^{237},e^{267},e^{345}+e^{147},e^{346}+e^{156}+e^{137}}\\
\hline
0,0,12,0,13+24,23,15+26+34 & e_{7}&[0,0,e^{12},0,e^{24}+e^{13},e^{23}]\\
\multicolumn{3}{|c|}{e^{123},e^{124},e^{125},e^{126},e^{134},e^{127}+e^{135},e^{136},e^{145},e^{146}-e^{127},e^{234},}\\
\multicolumn{3}{|c|}{e^{235},e^{236},e^{237}-e^{156}+e^{147},e^{127}+e^{245},e^{246},e^{256}+e^{137},-e^{247}+e^{345},e^{346}+e^{147}}\\
\hline
0,0,12,0,13,23+24,15+26 & e_{7}&[0,0,e^{12},0,e^{13},e^{23}+e^{24}]\\
\multicolumn{3}{|c|}{e^{123},e^{124},e^{125},e^{126},e^{127},e^{134},e^{135},e^{145},e^{136}+e^{146},e^{234},e^{235},e^{236},-e^{136}+e^{245},}\\
\multicolumn{3}{|c|}{e^{246},e^{247}+e^{237}-e^{156},e^{256}+e^{137},-e^{157}+e^{267},e^{237}-e^{156}+e^{345},e^{147}+e^{346}}\\
\hline
0,0,12,0,13,23+24,16+25+34 & e_{7}&[0,0,e^{12},0,e^{13},e^{24}+e^{23}]\\
\multicolumn{3}{|c|}{e^{123},e^{124},e^{125},e^{126},e^{134},e^{135},e^{127}+e^{136},e^{145},-e^{127}+e^{146},e^{137}+e^{147}+e^{156},}\\
\multicolumn{3}{|c|}{e^{234},e^{235},e^{236},e^{127}+e^{245},e^{246},e^{256}-e^{237},e^{345}+e^{147},e^{346}-e^{247}}\\
\hline
\end{array}\]
\end{table}

\begin{table}
\[\begin{array}{|lcr|}
\hline
0,0,12,13,23,14-25,15+24 & e_{7}&[0,0,e^{12},e^{13},e^{23},-e^{25}+e^{14}]\\
\multicolumn{3}{|c|}{e^{123},e^{124},e^{125},e^{126},e^{127},e^{134},e^{135},e^{145}+e^{137},e^{147},e^{234},e^{235},}\\
\multicolumn{3}{|c|}{e^{236}+e^{137},-e^{136}+e^{237},-e^{136}+e^{245},e^{156}+e^{246},e^{247}+e^{157}-2 e^{146},}\\
\multicolumn{3}{|c|}{e^{256}+e^{146},e^{257},-e^{146}+e^{345}}\\
\hline
0,0,0,12,14+23,13-24,15-34 & e_{7}&[0,0,0,e^{12},e^{14}+e^{23},e^{13}-e^{24}]\\
\multicolumn{3}{|c|}{e^{123},e^{124},e^{125},e^{126},e^{134},e^{135},e^{136}+e^{127},e^{137},e^{145}+e^{127},e^{146},}\\
\multicolumn{3}{|c|}{e^{147},e^{234},e^{235}+e^{127},e^{236},e^{245},e^{246}-e^{127},e^{345}+e^{237},-e^{247}+e^{346}-e^{156}}\\
\hline
0,0,12,0,23,24,13+25+46 & e_{7}&[0,0,e^{12},0,e^{23},e^{24}]\\
\multicolumn{3}{|c|}{e^{123},e^{124},e^{125},e^{126},e^{134},e^{135},e^{145}+e^{136},e^{146},e^{234},e^{235},e^{236},}\\
\multicolumn{3}{|c|}{e^{245},e^{246},e^{247}-e^{136},e^{256},e^{267}-e^{147}-e^{156},e^{345}+e^{147},e^{346}-e^{127},e^{237}+e^{456}}\\
\hline
0,0,12,0,13+24,23,15+34-26 & e_{7}&[0,0,e^{12},0,e^{24}+e^{13},e^{23}]\\
\multicolumn{3}{|c|}{e^{123},e^{124},e^{125},e^{126},e^{134},e^{135}+e^{127},e^{136},e^{145},e^{146}-e^{127},e^{234},}\\
\multicolumn{3}{|c|}{e^{235},e^{236},-e^{156}-e^{147}+e^{237},e^{245}+e^{127},e^{246},e^{256}-e^{137},e^{345}-e^{247},e^{346}-e^{147}}\\
\hline
0,0,12,0,13,23+24,15-26 & e_{7}&[0,0,e^{12},0,e^{13},e^{24}+e^{23}]\\
\multicolumn{3}{|c|}{e^{123},e^{124},e^{125},e^{126},e^{127},e^{134},e^{135},e^{145},e^{146}+e^{136},e^{234},e^{235},e^{236},e^{245}-e^{136},}\\
\multicolumn{3}{|c|}{e^{246},-e^{156}+e^{237}+e^{247},-e^{137}+e^{256},e^{157}+e^{267},-e^{156}+e^{345}+e^{237},-e^{147}+e^{346}}\\
\hline
\end{array}\]
\end{table}

\begin{table}
  \caption{\label{table:step5} Step 5 nilpotent Lie algebras of dimension $7$}
\[\begin{array}{|lcr|}
\hline
0,0,12,13,14,15,23 & e_{6}(P)&[0,0,e^{12},e^{13},e^{14},e^{23}]\\
\hline
0,0,12,13,14,25-34,23 & e_{7}(P)&[0,0,e^{12},e^{13},e^{14},e^{25}-e^{34}]\\
\hline
0,0,12,13,14,15,25-34 & e_{6}(P)&[0,0,e^{12},e^{13},e^{14},e^{25}-e^{34}]\\
\hline
0,0,12,13,14,15+23,25-34 & e_{6}(P)&[0,0,e^{12},e^{13},e^{14},-e^{34}+e^{25}]\\
\hline
0,0,12,13,14+23,15+24,23 & e_{6}(P)&[0,0,e^{12},e^{13},e^{23}+e^{14},e^{23}]\\
\hline
0,0,12,13,14+23,25-34,23 & e_{7}(P)&[0,0,e^{12},e^{13},e^{14}+e^{23},e^{25}-e^{34}]\\
\hline
0,0,12,13,14+23,15+24,25-34 & e_{6}(P)&[0,0,e^{12},e^{13},e^{23}+e^{14},-e^{34}+e^{25}]\\
\hline
0,0,12,13,14,0,15+26 & e_{7}&[0,0,e^{12},e^{13},e^{14},0]\\
\multicolumn{3}{|c|}{e^{123},e^{124},e^{125},e^{126},e^{127},e^{134},e^{135},e^{136},e^{145},e^{146},e^{156},e^{167},}\\
\multicolumn{3}{|c|}{e^{234},e^{236},-e^{237}+e^{245},e^{246}+e^{137},e^{256}+e^{147},-e^{157}+e^{267},e^{147}+e^{346}}\\
\hline
0,0,12,13,14,0,15+23+26 & e_{7}&[0,0,e^{12},e^{13},e^{14},0]\\
\multicolumn{3}{|c|}{e^{123},e^{124},e^{125},e^{126},e^{127},e^{134},e^{135},e^{136},e^{145},e^{146},e^{156},e^{137}+e^{167},e^{234},e^{236},}\\
\multicolumn{3}{|c|}{-e^{237}+e^{245},e^{137}+e^{246},e^{147}+e^{256}-e^{235},-e^{157}+e^{267}+e^{237},e^{147}-e^{235}+e^{346}}\\
\hline
0,0,12,13,14,0,16+25-34 & e_{7}&[0,0,e^{12},e^{13},e^{14},0]\\
\multicolumn{3}{|c|}{e^{123},e^{124},e^{125},e^{126},e^{134},e^{135},e^{136},e^{145},e^{146},e^{156},e^{234},e^{127}+e^{235},}\\
\multicolumn{3}{|c|}{e^{236},e^{137}+e^{245},-e^{237}+e^{246},e^{147}+e^{345},e^{346}-e^{256}}\\
\hline
0,0,12,13,14+23,0,15+24+26 & e_{7}&[0,0,e^{12},e^{13},e^{14}+e^{23},0]\\
\multicolumn{3}{|c|}{e^{123},e^{124},e^{125},e^{126},e^{127},e^{134},e^{135},e^{136},e^{146},-e^{137}+e^{156}-e^{145},}\\
\multicolumn{3}{|c|}{e^{167}+e^{147},e^{234},e^{235}-e^{145},e^{236},-e^{237}+e^{245},e^{137}+e^{145}+e^{246},e^{147}+e^{256},}\\
\multicolumn{3}{|c|}{-\frac{1}{2} e^{157}-\frac{1}{2} e^{247}+\frac{1}{2} e^{267}+e^{345},e^{147}+e^{346}}\\
\hline
\end{array}\]
\end{table}

\begin{table}
\[\begin{array}{|lcr|}
\hline
0,0,12,13,14+23,0,16+25-34 & e_{7}&[0,0,e^{12},e^{13},e^{23}+e^{14},0]\\
\multicolumn{3}{|c|}{e^{123},e^{124},e^{125},e^{126},e^{134},e^{135},e^{136},e^{145}+e^{127},e^{146},e^{234},e^{127}+e^{235},}\\
\multicolumn{3}{|c|}{e^{236},e^{237}+e^{156},e^{137}+e^{245},e^{246}+e^{156},e^{147}+e^{345},e^{346}-e^{256}}\\
\hline
0,0,12,13,14,23,15+26 & e_{7}(P)&[0,0,e^{12},e^{13},e^{14},e^{23}]\\
\hline
0,0,12,13,14,23,16+24+25-34 & e_{7}(P)&[0,0,e^{12},e^{13},e^{14},e^{23}]\\
\hline
0,0,12,13,14,23,15+25+26-34 & e_{7}(P)&[0,0,e^{12},e^{13},e^{14},e^{23}]\\
\hline
0,0,12,13,0,14+25,16+35 & e_{7}&[0,0,e^{12},e^{13},0,e^{25}+e^{14}]\\
\multicolumn{3}{|c|}{e^{123},e^{124},e^{125},e^{126},e^{134},e^{135},e^{127}+e^{136},e^{137},e^{145},e^{156},e^{157},}\\
\multicolumn{3}{|c|}{e^{234},e^{235},e^{245}-e^{127},-e^{237}+e^{246},e^{256}-e^{146},}\\
\multicolumn{3}{|c|}{e^{345}+e^{146},-e^{257}+e^{356}-e^{147},-\frac{1}{2} e^{167}-\frac{1}{2} e^{357}+e^{456}}\\
\hline
0,0,12,13,0,14+25,16+25+35 & e_{7}&[0,0,e^{12},e^{13},0,e^{25}+e^{14}]\\
\multicolumn{3}{|c|}{e^{123},e^{124},e^{125},e^{126},e^{134},e^{135},e^{136}+e^{127},e^{137}+e^{127},e^{145},e^{156},}\\
\multicolumn{3}{|c|}{e^{157},e^{234},e^{235},e^{245}-e^{127},e^{246}-e^{237},e^{256}-e^{146},e^{345}+e^{146},}\\
\multicolumn{3}{|c|}{-e^{257}+e^{356}+e^{146}-e^{147},\frac{1}{2} e^{257}-\frac{1}{2} e^{357}+e^{456}-\frac{1}{2} e^{167}}\\
\hline
0,0,12,13,0,14+25,26-34 & e_{7}&[0,0,e^{12},e^{13},0,e^{25}+e^{14}]\\
\multicolumn{3}{|c|}{e^{123},e^{124},e^{125},e^{126},e^{134},e^{135},e^{145},e^{156},e^{234},e^{235},e^{236}+e^{127},}\\
\multicolumn{3}{|c|}{e^{237},e^{245}+e^{136},e^{246}+e^{137},e^{256}-e^{146},e^{146}+e^{345},e^{147}+e^{257}+e^{346}}\\
\hline
0,0,12,13,0,14+25,15+26-34 & e_{7}&[0,0,e^{12},e^{13},0,e^{14}+e^{25}]\\
\multicolumn{3}{|c|}{e^{123},e^{124},e^{125},e^{126},e^{134},e^{135},e^{145},e^{156},e^{234},e^{235},e^{127}+e^{236},}\\
\multicolumn{3}{|c|}{e^{136}+e^{237},e^{136}+e^{245},e^{137}+e^{246},e^{256}-e^{146},e^{345}+e^{146},e^{257}+e^{346}+e^{147}}\\
\hline
0,0,12,13,0,14+23+25,16+24+35 & e_{7}&[0,0,e^{12},e^{13},0,e^{14}+e^{25}+e^{23}]\\
\multicolumn{3}{|c|}{e^{123},e^{124},e^{125},e^{126},e^{134},e^{135},e^{136}+e^{127},e^{145},e^{156}-e^{127},e^{157}+e^{137}+e^{146},}\\
\multicolumn{3}{|c|}{e^{234},e^{235},e^{137}+e^{236},e^{245}-e^{127},-e^{237}+e^{246},e^{256}-e^{137}-e^{146},}\\
\multicolumn{3}{|c|}{e^{345}+e^{137}+e^{146},-e^{257}-e^{147}+e^{356}}\\
\hline
0,0,12,13,0,14+23+25,26-34 & e_{7}&[0,0,e^{12},e^{13},0,e^{23}+e^{14}+e^{25}]\\
\multicolumn{3}{|c|}{e^{123},e^{124},e^{125},e^{126},e^{134},e^{135},e^{145},e^{156}+e^{136},e^{234},e^{235},e^{236}+e^{127},}\\
\multicolumn{3}{|c|}{e^{237},e^{245}+e^{136},e^{246}+e^{137},e^{256}-e^{146}-e^{127},e^{345}+e^{146}+e^{127},e^{346}+e^{257}+e^{147}}\\
\hline
0,0,12,13,0,14+23+25,15+26-34 & e_{7}&[0,0,e^{12},e^{13},0,e^{23}+e^{25}+e^{14}]\\
\multicolumn{3}{|c|}{e^{123},e^{124},e^{125},e^{126},e^{134},e^{135},e^{145},e^{136}+e^{156},e^{234},e^{235},e^{127}+e^{236},e^{136}+e^{237},}\\
\multicolumn{3}{|c|}{e^{136}+e^{245},e^{137}+e^{246},-e^{127}+e^{256}-e^{146},e^{345}+e^{127}+e^{146},e^{346}+e^{147}+e^{257}}\\
\hline
0,0,12,13,23,15+24,16+34& \text{ resists }&\\
\multicolumn{3}{|c|}{e^{123},e^{124},e^{125},e^{126},e^{134},e^{135},e^{127}+e^{136},e^{137},-e^{127}+e^{145},e^{146},}\\
\multicolumn{3}{|c|}{e^{147},e^{234},e^{235},e^{245}-e^{236},e^{246}-e^{156}-2 e^{237},e^{256},e^{345}-e^{156}-e^{237},}\\
\multicolumn{3}{|c|}{e^{346}-e^{247}-e^{157},e^{267}-e^{357}+e^{456}}\\
\hline
0,0,12,13,23,15+24,16+25+34& \text{ resists }&\\
\multicolumn{3}{|c|}{e^{123},e^{124},e^{125},e^{126},e^{134},e^{135},e^{136}+e^{127},e^{145}-e^{127},e^{146},e^{234},}\\
\multicolumn{3}{|c|}{e^{235},e^{236}+e^{137},e^{237}+e^{147}+e^{156},e^{137}+e^{245},2 e^{147}+e^{156}+e^{246},e^{256},}\\
\multicolumn{3}{|c|}{e^{345}+e^{147},e^{346}-e^{247}-e^{157},e^{456}-e^{357}+e^{267}}\\
\hline
0,0,12,13,23,15+24,16+14+25+34 & e_{7}&[0,0,e^{12},e^{13},e^{23},e^{24}+e^{15}]\\
\multicolumn{3}{|c|}{e^{123},e^{124},e^{125},e^{126},e^{134},e^{135},e^{136}+e^{127},e^{145}-e^{127},e^{146},e^{234},}\\
\multicolumn{3}{|c|}{e^{235},e^{236}+e^{137},e^{237}+e^{156}+e^{147}-e^{127},e^{245}+e^{137},e^{246}+e^{156}+2 e^{147},}\\
\multicolumn{3}{|c|}{e^{256},e^{345}+e^{147},e^{346}-e^{247}-e^{157}}\\
\hline
\end{array}\]
\end{table}

\clearpage

\begin{table}[h]
 \[\begin{array}{|lcr|}
\hline
0,0,12,13,23,15+24,16+14+34 & e_{7}&[0,0,e^{12},e^{13},e^{23},e^{24}+e^{15}]\\
\multicolumn{3}{|c|}{e^{123},e^{124},e^{125},e^{126},e^{134},e^{135},e^{127}+e^{136},e^{137},-e^{127}+e^{145},e^{146},}\\
\multicolumn{3}{|c|}{e^{147},e^{234},e^{235},-e^{236}+e^{245},2 e^{127}-2 e^{237}-e^{156}+e^{246},}\\
\multicolumn{3}{|c|}{e^{256},e^{127}-e^{237}-e^{156}+e^{345},-e^{247}-e^{157}+e^{346}}\\
\hline
0,0,12,13,23,15+24,16+26+34-35 & e_{7}&[0,0,e^{12},e^{13},e^{23},e^{24}+e^{15}]\\
\multicolumn{3}{|c|}{e^{123},e^{124},e^{125},e^{126},e^{134},e^{135},e^{136}+e^{145},e^{146},e^{234},e^{235},e^{127}+e^{136}+e^{236},}\\
\multicolumn{3}{|c|}{e^{137}+e^{237},e^{127}+e^{136}+e^{245},e^{246}+2 e^{137}-e^{156},}\\
\multicolumn{3}{|c|}{e^{256},e^{345}+e^{137}-e^{156},-e^{346}+e^{356}+e^{247}+e^{157}}\\
\hline
0,0,0,12,14+23,15-34,16-35 & e_{7}&[0,0,0,e^{12},e^{23}+e^{14},-e^{34}+e^{15}]\\
\multicolumn{3}{|c|}{e^{123},e^{124},e^{125},e^{134},e^{135},e^{136},e^{137},e^{126}+e^{145},e^{146},e^{234},e^{126}+e^{235},e^{127}+e^{236},}\\
\multicolumn{3}{|c|}{e^{237}+e^{147}+e^{156},e^{245},-e^{127}+e^{345},-2 e^{147}+e^{346}-e^{156},-e^{347}+e^{356}+e^{157}}\\
\hline
0,0,0,12,14+23,15-34,16+23-35 & e_{7}&[0,0,0,e^{12},e^{23}+e^{14},e^{15}-e^{34}]\\
\multicolumn{3}{|c|}{e^{123},e^{124},e^{125},e^{134},e^{135},e^{136},e^{137},e^{145}+e^{126},e^{146},e^{234},e^{235}+e^{126},}\\
\multicolumn{3}{|c|}{e^{236}+e^{127},e^{237}+e^{126}+e^{156}+e^{147},e^{245},e^{345}-e^{127},}\\
\multicolumn{3}{|c|}{e^{346}-2 e^{126}-e^{156}-2 e^{147},e^{356}-e^{347}+e^{127}+e^{157}}\\
\hline
0,0,0,12,14+23,15-34,16+24-35 & e_{7}&[0,0,0,e^{12},e^{14}+e^{23},e^{15}-e^{34}]\\
\multicolumn{3}{|c|}{e^{123},e^{124},e^{125},e^{134},e^{135},e^{136},-e^{126}+e^{137},e^{126}+e^{145},e^{146},}\\
\multicolumn{3}{|c|}{e^{234},e^{126}+e^{235},e^{236}+e^{127},e^{156}+e^{147}+e^{237},e^{245},}\\
\multicolumn{3}{|c|}{e^{345}-e^{127},-e^{156}+e^{346}-2 e^{147},-e^{246}+e^{356}-e^{347}+e^{157}}\\
\hline
0,0,12,13,23,24+15,16+14-25+34 & e_{7}&[0,0,e^{12},e^{13},e^{23},e^{24}+e^{15}]\\
\multicolumn{3}{|c|}{e^{123},e^{124},e^{125},e^{126},e^{134},e^{135},e^{127}+e^{136},-e^{127}+e^{145},e^{146},e^{234},}\\
\multicolumn{3}{|c|}{e^{235},e^{236}-e^{137},-e^{127}+e^{237}-e^{147}+e^{156},-e^{137}+e^{245},-2 e^{147}+e^{156}+e^{246},}\\
\multicolumn{3}{|c|}{e^{256},-e^{147}+e^{345},-e^{247}-e^{157}+e^{346}}\\
\hline
0,0,12,13,23,-14-25,16-35 & e_{7}(P)&[0,0,e^{12},e^{13},e^{23},-e^{14}-e^{25}]\\
\hline
0,0,12,13,23,-14-25,16-35+25 & e_{7}(P)&[0,0,e^{12},e^{13},e^{23},-e^{25}-e^{14}]\\
\hline
0,0,0,12,14+23,15-34,16-23-35 & e_{7}&[0,0,0,e^{12},e^{23}+e^{14},e^{15}-e^{34}]\\
\multicolumn{3}{|c|}{e^{123},e^{124},e^{125},e^{134},e^{135},e^{136},e^{137},e^{145}+e^{126},e^{146},e^{234},e^{235}+e^{126},}\\
\multicolumn{3}{|c|}{e^{236}+e^{127},e^{237}-e^{126}+e^{156}+e^{147},e^{245},}\\
\multicolumn{3}{|c|}{e^{345}-e^{127},e^{346}+2 e^{126}-e^{156}-2 e^{147},e^{356}-e^{347}-e^{127}+e^{157}}\\
\hline
\end{array}\]
\end{table}

\bigskip

\noindent {\bf Acknowledgments.} This work has been partially
supported through Project MICINN (Spain) MTM2008-06540-C02-01.

\smallskip

{\small

\small\noindent Dipartimento di Matematica e Applicazioni, Universit\`a di Milano Bicocca,  Via Cozzi 53, 20125 Milano, Italy.\\
\texttt{diego.conti@unimib.it}

\smallskip
\small\noindent Universidad del Pa\'{\i}s Vasco, Facultad de Ciencia y Tecnolog\'{\i}a, Departamento de Matem\'aticas,
Apartado 644, 48080 Bilbao, Spain. \\
\texttt{marisa.fernandez@ehu.es}\\

 \end{document}